\documentclass[11pt, reqno]{amsart}

\usepackage{textcmds} 
\usepackage{amsmath, amssymb, amsfonts, amstext, verbatim, amsthm, mathrsfs}
\usepackage[mathcal]{eucal}
\usepackage{microtype}
\usepackage[all]{xy}
\usepackage[modulo]{lineno}
\usepackage[usenames]{color}
\usepackage{aliascnt}
\usepackage{enumitem}
\usepackage{xspace}
\usepackage{amsfonts}
\usepackage{amssymb}
\usepackage[centertags]{amsmath}
\usepackage{amsthm}
\usepackage[margin=3cm]{geometry}
\usepackage{dsfont}
\usepackage{bm}
\usepackage{color}
\usepackage{subfigure}
\usepackage{amsmath}
\usepackage{array}
\usepackage[all]{xy}

\usepackage{graphics,graphicx} 

\usepackage{xfrac}

\let\oldtocsection=\tocsection

\let\oldtocsubsection=\tocsubsection

\renewcommand{\tocsection}[2]{\hspace{0em}\oldtocsection{#1}{#2}}
\renewcommand{\tocsubsection}[2]{\hspace{1em}\oldtocsubsection{#1}{#2}}

\usepackage[bookmarks=false,backref,colorlinks=true,linkcolor=blue,citecolor=red,urlcolor=blue,citebordercolor={0 0 1},urlbordercolor={0 0 1},linkbordercolor={0 0 1}]{hyperref} 

\newcommand{\C}{\mathbb{C}}

\newcommand{\R}{\mathbb{R}}
\newcommand{\Q}{\mathbb{Q}}
\newcommand{\Z}{\mathbb{Z}}

\newcommand{\CP}{\mathbb{C}P}
\newcommand{\PxP}{\mathbb{C}P^1 \times \mathbb{C}P^1}
\newcommand{\Blk}{\mathbb{C}P^2\# k\overline{\mathbb{C}P^2}}
\newcommand{\BlI}{\mathbb{C}P^2\# \overline{\mathbb{C}P^2}}
\newcommand{\BlII}{\mathbb{C}P^2\# 2\overline{\mathbb{C}P^2}}
\newcommand{\BlIII}{\mathbb{C}P^2\# 3\overline{\mathbb{C}P^2}}
\newcommand{\BlIV}{\mathbb{C}P^2\# 4\overline{\mathbb{C}P^2}}
\newcommand{\BlV}{\mathbb{C}P^2\# 5\overline{\mathbb{C}P^2}}
\newcommand{\BlVI}{\mathbb{C}P^2\# 6\overline{\mathbb{C}P^2}}
\newcommand{\BlVII}{\mathbb{C}P^2\# 7\overline{\mathbb{C}P^2}}
\newcommand{\BlVIII}{\mathbb{C}P^2\# 8\overline{\mathbb{C}P^2}}

\newcommand{\kabc}{k_1a^2,k_2b^2,k_3c^2}

\newcommand{\Tpqr}{\Theta^{n_1,n_2,n_3}_{p,q,r}}

\newcommand{\del}{\partial}
\newcommand{\bu}{\textbf{u}}
\newcommand{\bv}{\textbf{v}}
\newcommand{\bw}{\textbf{w}}

\newcommand{\subscript}[2]{$#1 _ #2$}

\numberwithin{equation}{section}

\newtheorem{thm}{Theorem}[section]
\newtheorem*{thm*}{Theorem}
\newtheorem{prp}[thm]{Proposition}
\newtheorem{lem}[thm]{Lemma}
\newtheorem{cor}[thm]{Corollary}

\newtheorem{cnj}[thm]{Conjecture}
\theoremstyle{definition}
\newtheorem{dfn}[thm]{Definition}
\newtheorem{rmk}[thm]{Remark}
\newtheorem{qu}[thm]{Question}
\newtheorem{clm}[thm]{Claim}

\numberwithin{equation}{section}

\begingroup 
\makeatletter 
\@for\theoremstyle:=definition,remark,plain,TheoremNum\do{%
\expandafter\g@addto@macro\csname th@\theoremstyle\endcsname{%
\addtolength\thm@preskip\parskip 
}%
} 
\endgroup 

\title{Infinitely many monotone Lagrangian tori in del Pezzo surfaces}

\author{Renato Vianna}
\thanks{The author is supported by the Herschel Smith postdoctoral fellowship from
the University of Cambridge.}

\address{
Renato Vianna\\
Centre for Mathematical Sciences\\
University of Cambridge\\
Cambridge, CB3 0WB\\
United Kingdom}
\email{r.vianna@dpmms.cam.ac.uk}

\begin{document}
  
  \begin{abstract} 
   
We construct almost toric fibrations (ATFs) on all del Pezzo surfaces, endowed
with a monotone symplectic form. Except for $\BlI$, $\BlII$, we are able to get
almost toric base diagrams (ATBDs) of triangular shape and prove the existence
of infinitely many symplectomorphism (in particular Hamiltonian isotopy) classes
of monotone Lagrangian tori in $\Blk$ for $k=0,3,4,5,6,7,8$. We name these tori
$\Theta^{n_1,n_2,n_3}_{p,q,r}$. Using the work of Karpov-Nogin, we are able to
classify all ATBDs of triangular shape. We are able to prove that $\BlI$ also
have infinitely many monotone Lagrangian tori up to symplectomorphism and we
conjecture that the same holds for $\BlII$. Finally, the Lagrangian tori
$\Theta^{n_1,n_2,n_3}_{p,q,r} \subset X$ can be seen as monotone fibers of ATFs,
such that, over its edge lies a fixed anticanonical symplectic torus $\Sigma$. We
argue that $\Theta^{n_1,n_2,n_3}_{p,q,r}$ give rise to infinitely many exact
Lagrangian tori in $X \setminus \Sigma$, even after attaching the positive end
of a symplectization to $\del(X \setminus \Sigma)$.

\end{abstract}

\maketitle

\vspace{-0.5cm}

\tableofcontents

\newpage

\section{Introduction} \label{sec: Intro}

We say that two Lagrangians submanifolds of a symplectic manifold $X$ belong to
the same {\it symplectomorphism class} if there is a symplectomorphism of $X$ sending
one Lagrangian to the other. Similar for {\it Hamiltonian isotopy class}. 

In \cite{Vi14}, it is explained how to get infinitely many symplectomorphism classes of
monotone Lagrangian tori in $\CP^2$. The idea is to construct different almost
toric fibrations (denoted from here by ATF) \cite{Sy03,SyLe10} of $\CP^2$. The
procedure starts by applying \emph{nodal trades} \cite{Sy03,SyLe10} to the
corners of the moment polytope and subsequently applying a series of \emph{nodal
slides} \cite{Sy03,SyLe10} through the monotone fiber of the ATF. In that way we
obtain infinitely many monotone Lagrangian tori $T(a^2,b^2,c^2)$ as the central
fibre of some almost toric base diagram (denoted from here by ATBD) describing
an ATF. They are indexed by Markov triples $(a,b,c)$, i.e., positive integer
solutions of $a^2 + b^2 + c^2 = 3abc$. We refer the reader to \cite{Vi14} for a
detailed account.

It is clear that the technique to construct ATFs and potentially get infinitely
many symplectomorphism classes of monotone Lagrangian tori would also apply for any
monotone toric symplectic 4-manifolds. In this paper we show that we can
actually construct ATFs in all monotone del Pezzo surfaces,i.e. , $\PxP$ and
$\Blk$, $0 \le k \le 8$. 

\begin{figure}[h!]   
  
\begin{center} 
  
\begin{minipage}{0.4\textwidth}
\vspace{0.6cm}
\centerline{\includegraphics[scale=0.5]{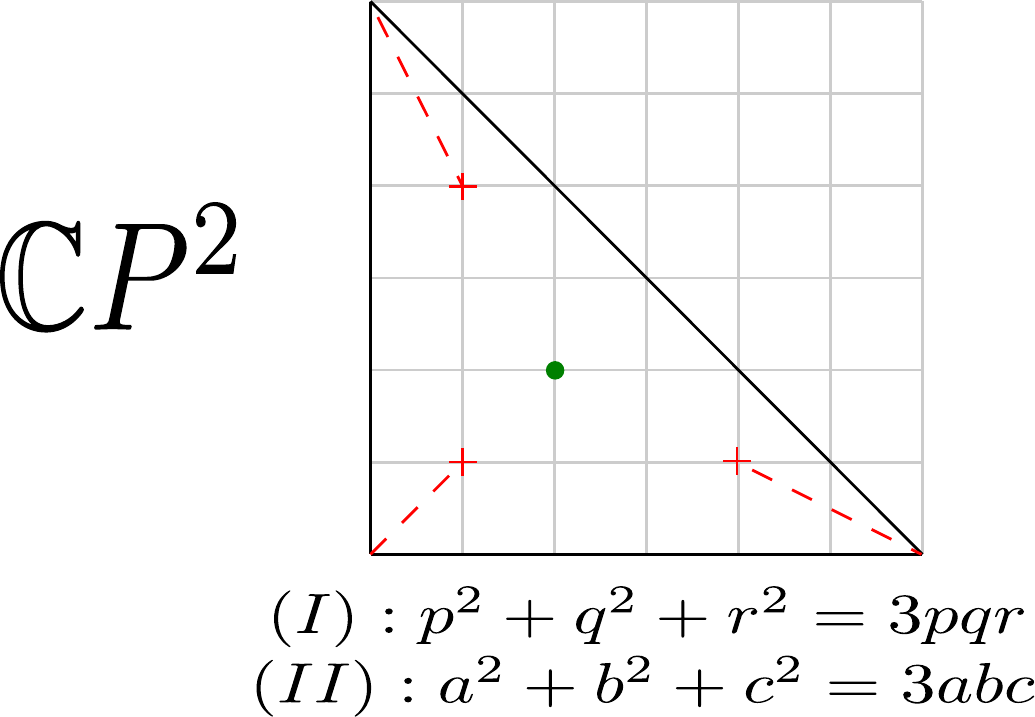}}
\vspace{1cm}
\centerline{\includegraphics[scale=0.5]{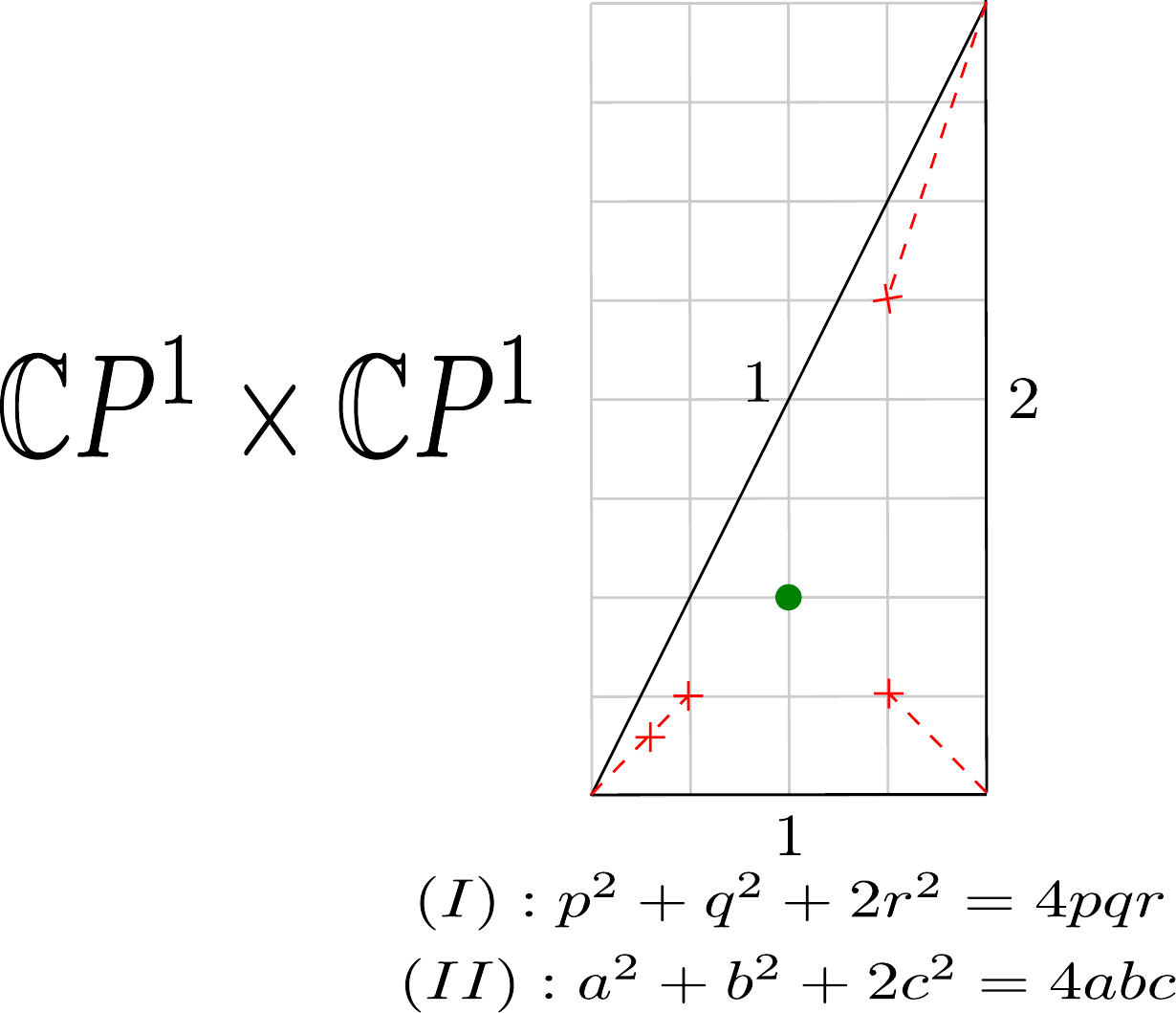}} 
\end{minipage}
\hspace{1cm}
\begin{minipage}{0.4\textwidth}
\centerline{\includegraphics[scale=0.5]{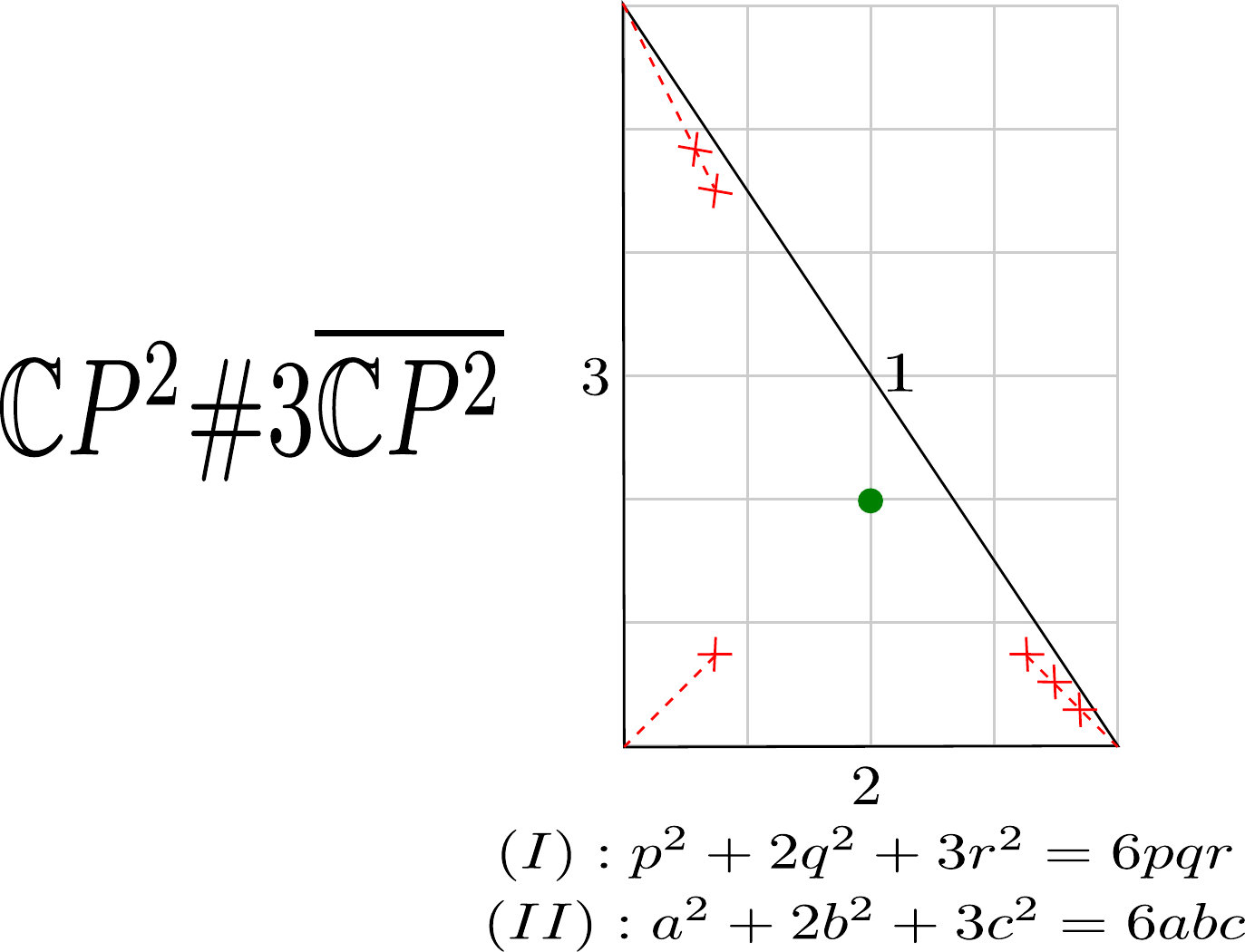}}
\vspace{0.8cm}
\centerline{\includegraphics[scale=0.5]{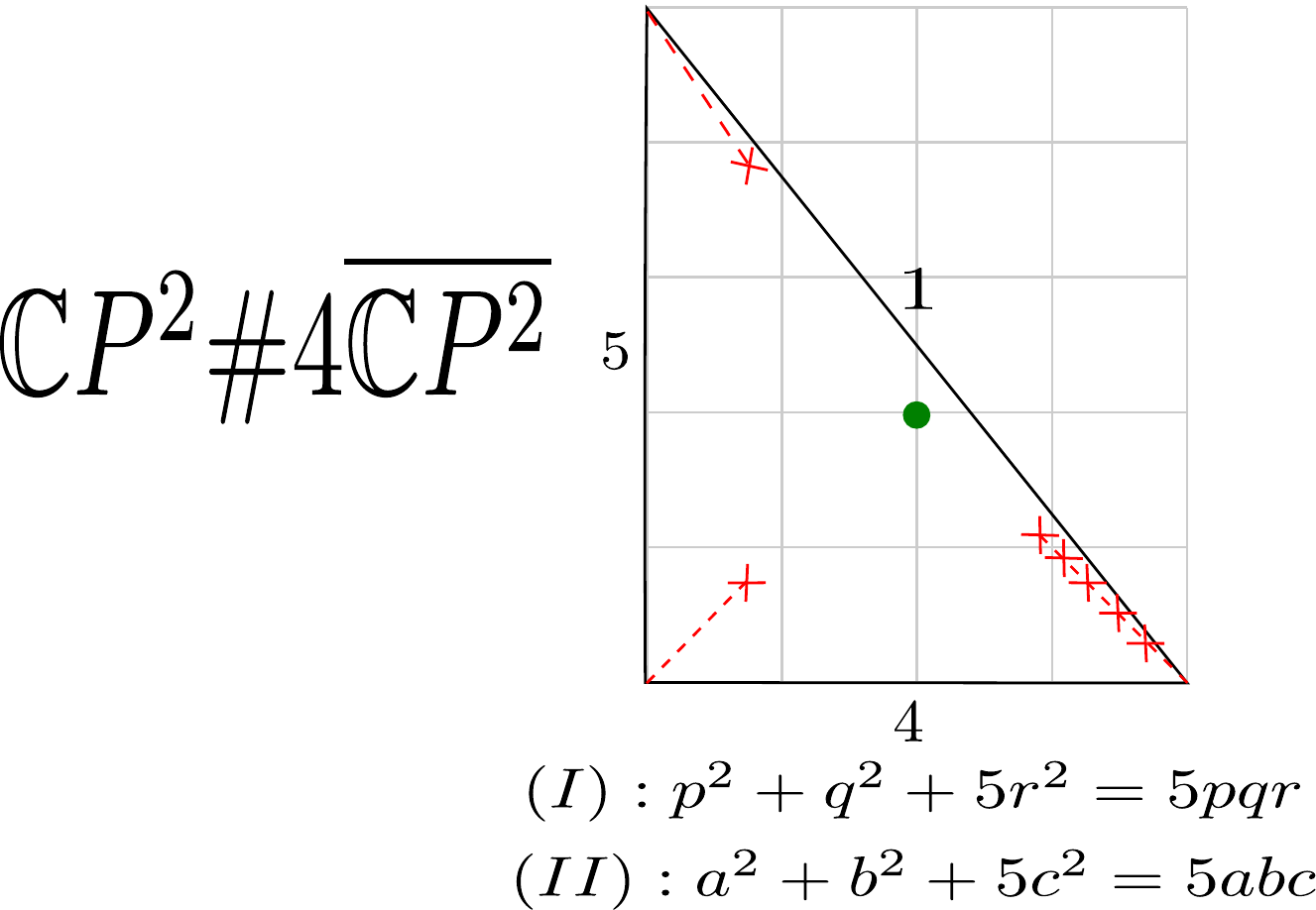}}
\end{minipage}

\caption{ATBDs of $\CP^2$, $\PxP$, $\BlII$, $\BlIV$ of triangular shape.
Near each base diagram is the corresponding Markov type I and II equations
(Definition \ref{dfn: ATBDequation}).}
\label{fig: triangshape1} 

\end{center} 
\end{figure} 

\begin{figure}[h!]   
  
\begin{center}

\begin{minipage}{0.3\textwidth}
\centerline{\includegraphics[scale=0.45]{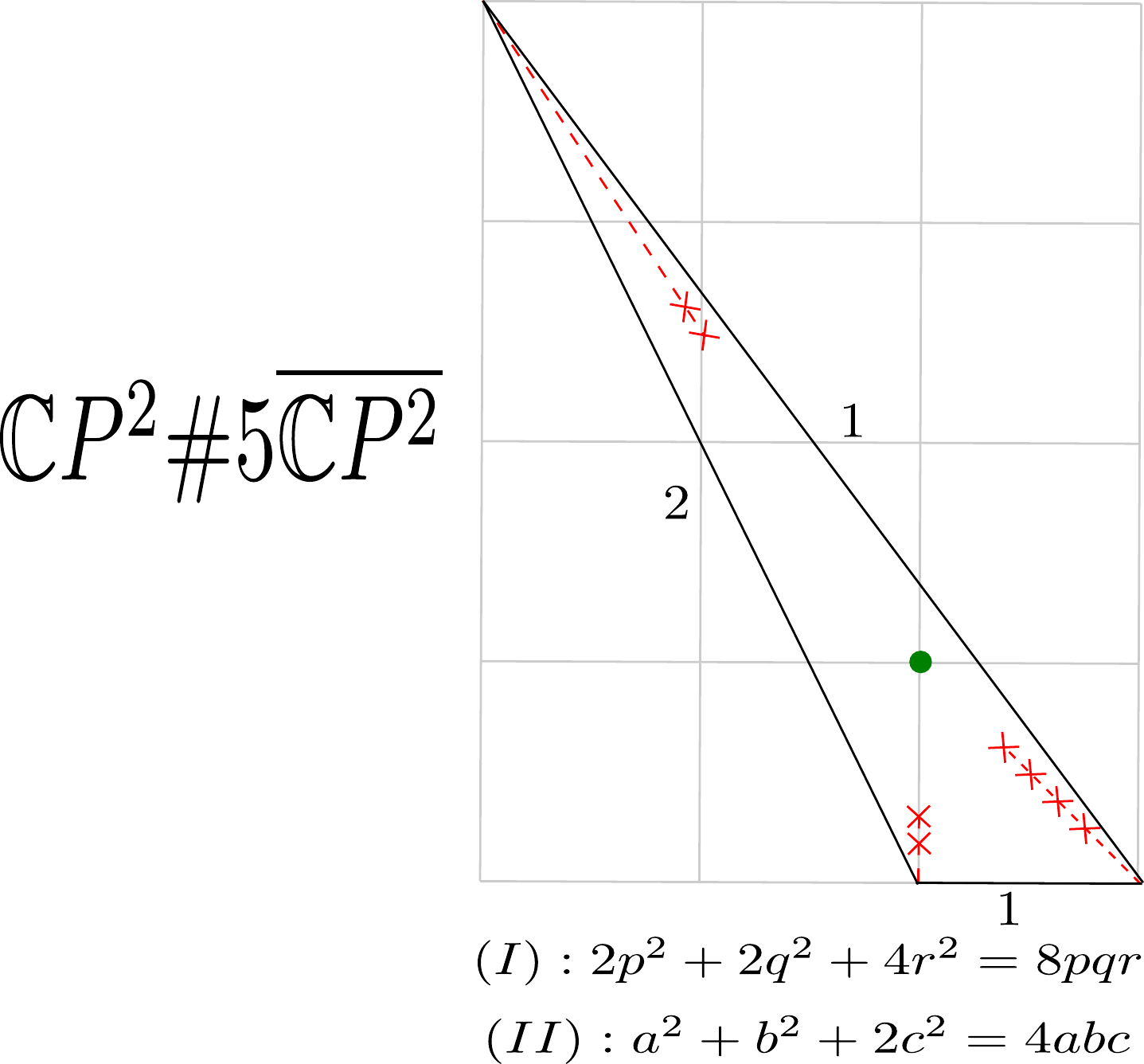}}
\end{minipage}
\hspace{1.5cm}
\begin{minipage}{0.4\textwidth}
\centerline{\includegraphics[scale=0.45]{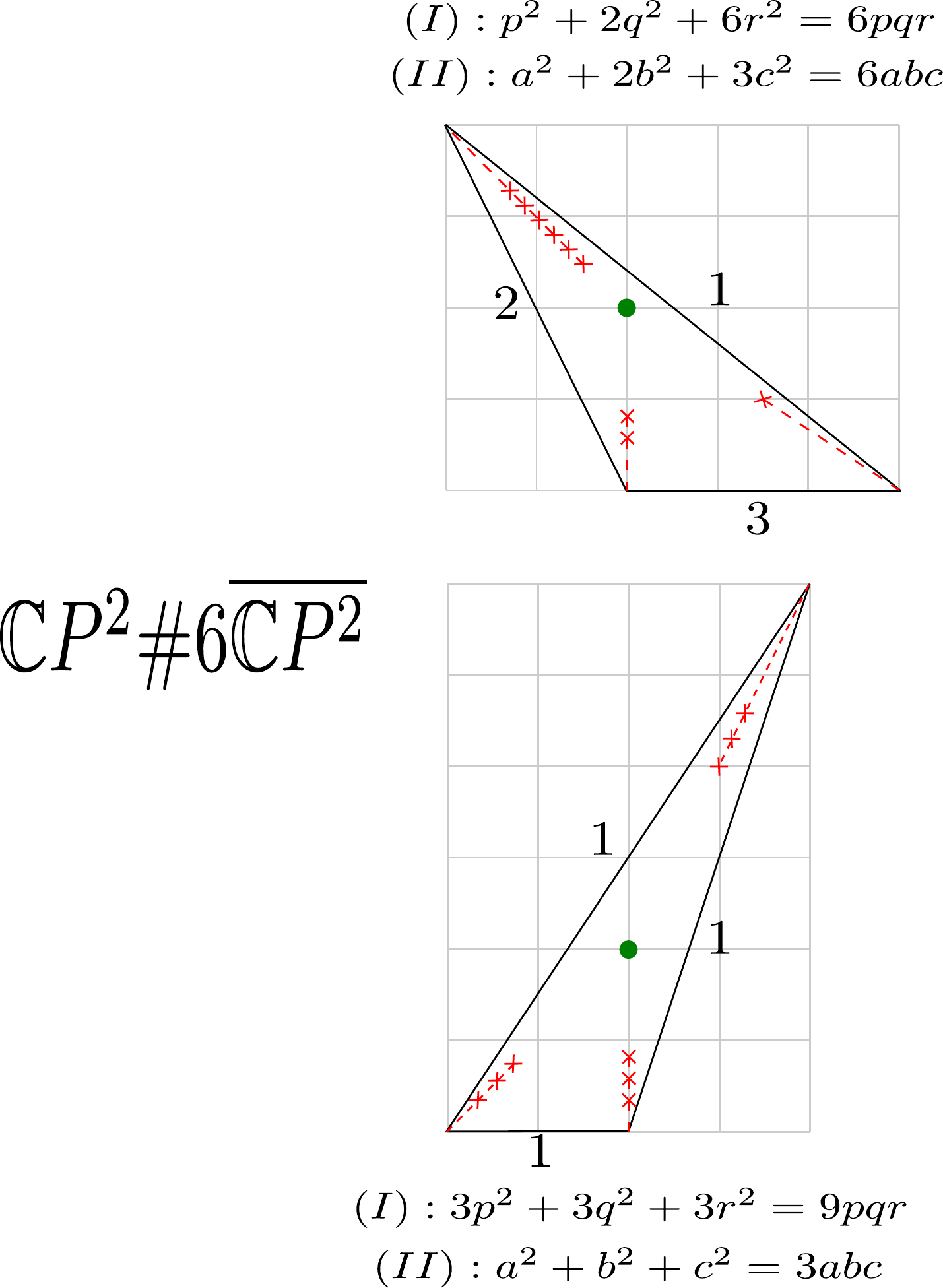}}
\end{minipage}

\caption{ATBDs of $\BlV$, $\BlVI$ of triangular shape.
Below each base diagram is the corresponding Markov type I and II equations.}
\label{fig: triangshape2} 

\end{center} 
\end{figure}

\begin{figure}[h!]   
  
\begin{center}

\centerline{\includegraphics[scale=0.5]{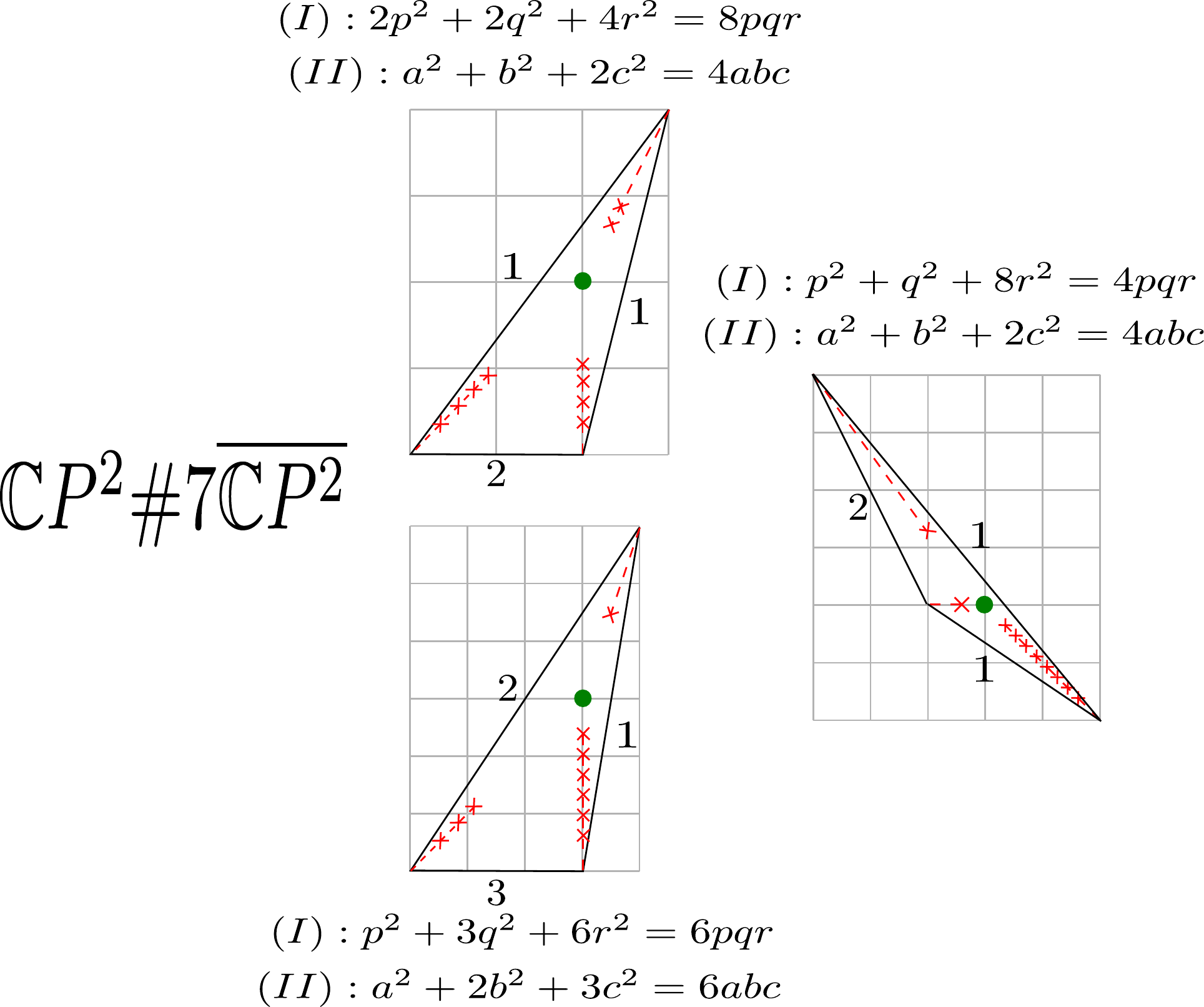}}

\caption{ATBDs of $\BlVII$ of triangular shape.
Close to each base diagram is the corresponding Markov type I and II equations.}
\label{fig: triangshape3} 

\end{center} 
\end{figure}

\begin{figure}[h!]   
  
\begin{center}

\centerline{\includegraphics[scale=0.5]{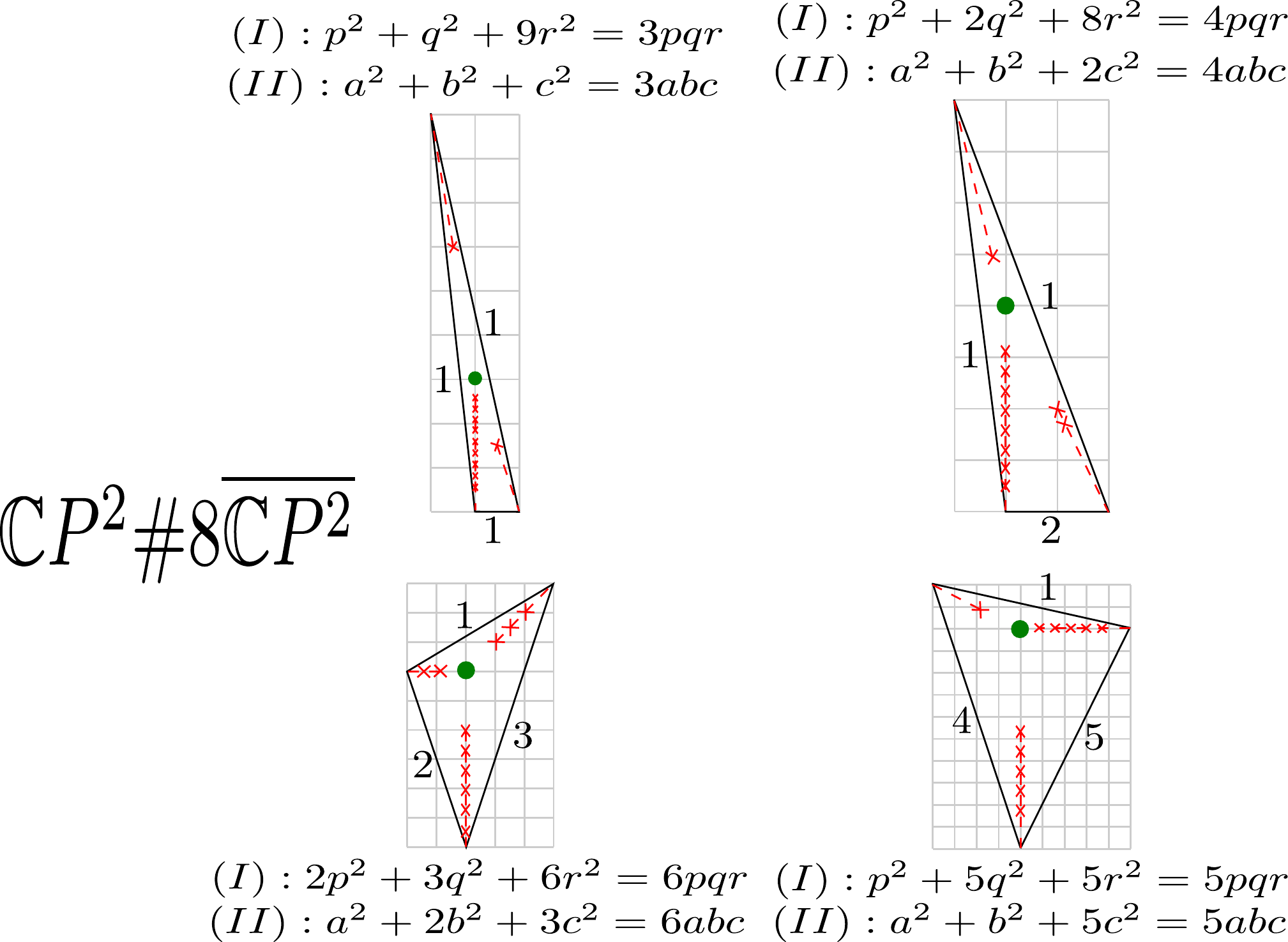}}

\caption{ATBDs of $\BlVIII$ of triangular shape.
Close to each base diagram is the corresponding Markov type I and II equations.}
\label{fig: triangshape4} 

\end{center} 
\end{figure}

In \cite{OhtaOno96,OhtaOno97}, Ohta-Ono proved that the diffeomorphism type of
any closed monotone symplectic 4-manifold is $\PxP$ and $\Blk$, $0 \le k \le 8$,
based on the work of McDuff \cite{MD90} and Taubes \cite{Ta95,Ta96,Ta00Book}.
Also, in \cite{MD96}, McDuff showed that uniqueness of blowups (of given sizes)
for 4-manifolds of non-simple Seiberg-Witten type (which includes $\CP^2$).
Finally, from the uniqueness of the monotone symplectic form for $\CP^2$ and
$\PxP$ \cite{Gr85,MD90,Ta95,Ta00Book} (see also the excelent survey
\cite{Sa13Survey}), we get uniqueness of monotone symplectic structures on
$\PxP$ and $\Blk$, $0 \le k \le 8$. We refer the later as del Pezzo surfaces
thinking of them endowed with a monotone symplectic form.

The algebraic count of Maslov index 2 pseudo-holomorphic disks with boundary in
a monotone Lagrangian $L$ and relative homotopy class $\beta$ is an invariant of
the symplectomorphism class of $L$, first pointed out in \cite{ElPo93} (we
believe that the formal proof needs the work of \cite{La00}). To distinguish the
monotone Lagrangian tori we built in $\CP^2$, we used an invariant among
monotone Lagrangian $L$ in the same symplectomorphism class based on the above count
\cite{Vi14}. We named it the {\it boundary Maslov-2 convex hull} $\mho_L$
\cite[Section~4]{Vi14}, which is the convex hull in $H_1(L)$ of the set formed
by the boundary of Maslov index 2 classes in $\pi_2(X,L)$ that have non-zero
enumerative geometry. In other words, it is the convex hull for the Newton
polytope of the superpotential function - for definition of the superpotential
we refer the reader to \cite{Au07, Au09, FO3Book}.

For computing the above invariant, we employed the neck-stretching technique
\cite{EliGiHo10,CompSFT03} to get a degenerated limit of pseudo-holomorphic
disks with boundary in $T(a^2,b^2,c^2)$ inside the weighted projective space
$\CP(a^2,b^2,c^2)$. We then used positivity of intersection for orbifold disks
\cite{ChenRu02,Chen04} in the weighted projective space $\CP(a^2,b^2,c^2)$,
together with the computation of holomorphic disks away from the orbifold points
\cite{ChoPo14} to be able to compute $\mho_{T(a^2,b^2,c^2)}$, the boundary
Maslov-2 convex hull for $T(a^2,b^2,c^2)$. It follows that
$\mho_{T(a^2,b^2,c^2)}$ is incongruent (not related via $SL(2,\Z)$ upon a choice
of basis for the respectives 1st homotopy groups) to $\mho_{T(d^2,e^2,f^2)}$, if
$\{a,b,c\} \ne \{d,e,f\}$. 

The aim of this paper is to prove:

\begin{thm} \label{thm: main}
  There are infinitely many symplectomorphism classes of
  monotone Lagrangian tori inside:
\begin{enumerate}[label=(\alph*)]
   \item $\PxP$ and $\Blk$, $k= 0,3,4,5,6,7,8$; \label{thm: a}
   \item  $\BlI$;   \label{thm: b}
  \end{enumerate} 
\end{thm}

The proof of items \ref{thm: a}, \ref{thm: b} of the above theorem
differ a little. 

To prove Theorem \ref{thm: main} \ref{thm: a} we show that for $\PxP$ and
$\Blk$, $k= 0,3,4,5,6,7,8$ we can build almost toric base diagrams of triangular
shape described in Figures \ref{fig: triangshape1}, \ref{fig:
triangshape2}, \ref{fig: triangshape3}, \ref{fig: triangshape4}.      

We call a {\it mutation} in an ATBD with respect to a node $s$ if we apply a
{\it nodal slide operation} \cite{Sy03,SyLe10} (we always slide the node to pass
through the monotone fiber) together with a {\it transferring the cut operation}
\cite{Vi14}, with respect to $s$. 

The affine lengths of the edges of the ATBDs depicted in Figures \ref{fig:
triangshape1}-\ref{fig: triangshape4} are related to solutions of Markov type II
equations of the form

\begin{equation*} 
k_1a^2 + k_2b^2 + k_3c^2 = Kk_1k_2k_3abc. \ \ \eqref{eq: Markovtype}
\end{equation*}

We can apply a mutation $(a,b,c) \to (a' = Kk_2k_3bc - a, b, c)$ to obtain a new
solution of the same Markov type II equation. 

\begin{figure}[h!]  
\begin{center}
  
\centerline{\includegraphics[scale=0.57]{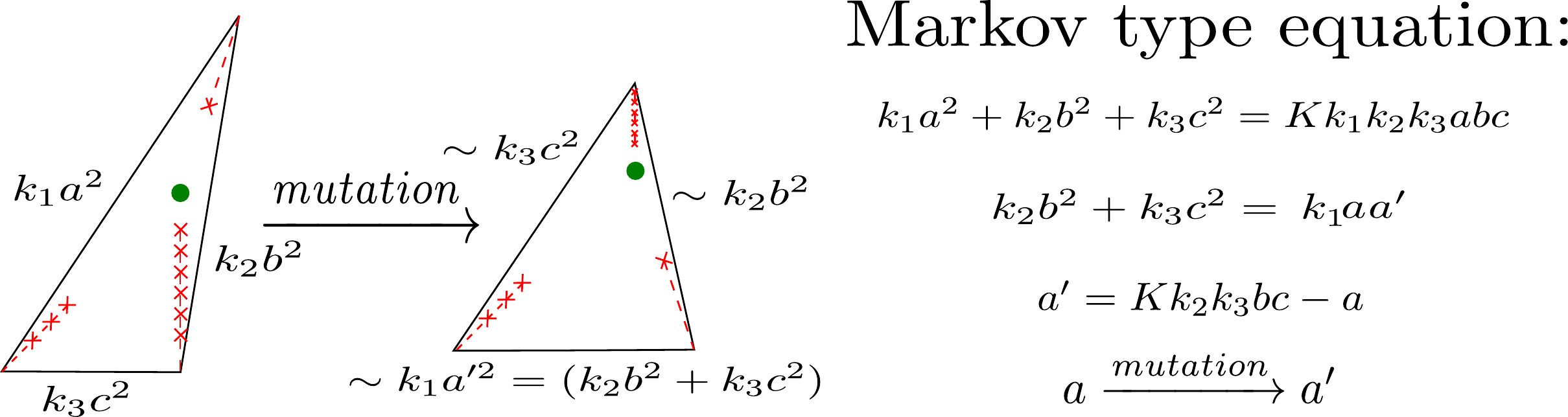}} 
  
\caption{Mutation on ATBD of triangular shape corresponds to mutation
on Markov type II triples.}
\label{fig: LemMut}   
\end{center}  
\end{figure}

Suppose we have an ATBD related to the Markov type II equation \eqref{eq:
Markovtype}, for some $K,k_1,k_2,k_3$. We prove on Section \ref{sec: Mutation}
(Lemma \ref{lem: Mutation}) that a mutation on an ATBD with respect to all nodes
in the same cut corresponds to a mutation on the respective Markov type II
triple solution of \eqref{eq: Markovtype}, as ilustrated in Figure \ref{fig:
LemMut}.

But the affine lengths do not determine the ATBDs of triangular shape, see
Figure \ref{fig: triangshape3}. What does is the node type (Definition \ref{dfn:
nodetype}). An ATBD of node-type $((n_1,p),(n_2,q),(n_3,r))$, must have
$(p,q,r)$ satisfying the Markov type I equation:

\begin{equation*}
 n_1p^2 + n_2q^2 + n_3r^2 = \sqrt{dn_1n_2n_3}pqr,  \ \ \eqref{eq: MarkovI}
  \end{equation*}

We name $\Tpqr$ the monotone fiber inside an ATBD of node-type
$((n_1,p),(n_2,q),(n_3,r))$ (we are assuming that the fiber lives on the
complement of all the cuts). We can aplly the same ideas of the proof of Theorem
1.1 of \cite{Vi14}[Section~4] to compute $\mho_{\Tpqr}$, the boundary Maslov-2
convex hull for each $\Tpqr$. 

Let's call the {\it limit orbifold} (Definition \ref{dfn: limorb}) of an ATBD
the orbifold described by the moment polytope given by deleting the cuts of the
ATBD (here we assume that the cuts are all in the eigendirection of the
monodromy around the respective node). Informally, we think that we nodal slide
all the nodes of the ATBD towards the edge, so in the limit the described by the
corresponding ATF is ``degenerating'' to the limit orbifold.
 
In the proof of Theorem \ref{thm: main} \ref{thm: a}, we look at degenerated
limit of pseudo-holomorphic disks with boundary in $T(\kabc)$, which lives in
the limit orbifold of the corresponding ATBD. One important aspect we use to
compute $\mho_{\Tpqr}$ is positivity of intersection between the degenerated
limit of pseudo-holomorphic curves in the limit orbifold and pre-image of the
edges of the limit orbifold's moment polytope. We may loose this property if the
moment polytope of the limit orbifold is not a triangle.

For the cases $\Blk$, $k= 1,2$, we can also construct infinitely many ATBDs,
each one describing an ATF with a monotone Lagrangian torus fiber, for instance
the ones in Figures \ref{fig: ATBD1Blup},\ref{fig: ATBD2Blup}. Let's name
$T_1(a,b)$ the monotone torus fiber of the ATF of $\BlI$ depicted in Figure
\ref{fig: ATBD1Blup}, similar $T_2(a,b) \subset \BlII$ depicted in Figure
\ref{fig: ATBD2Blup}.

\begin{figure}[h!]   
  
\begin{center} 
  
\centerline{\includegraphics[scale=0.4]{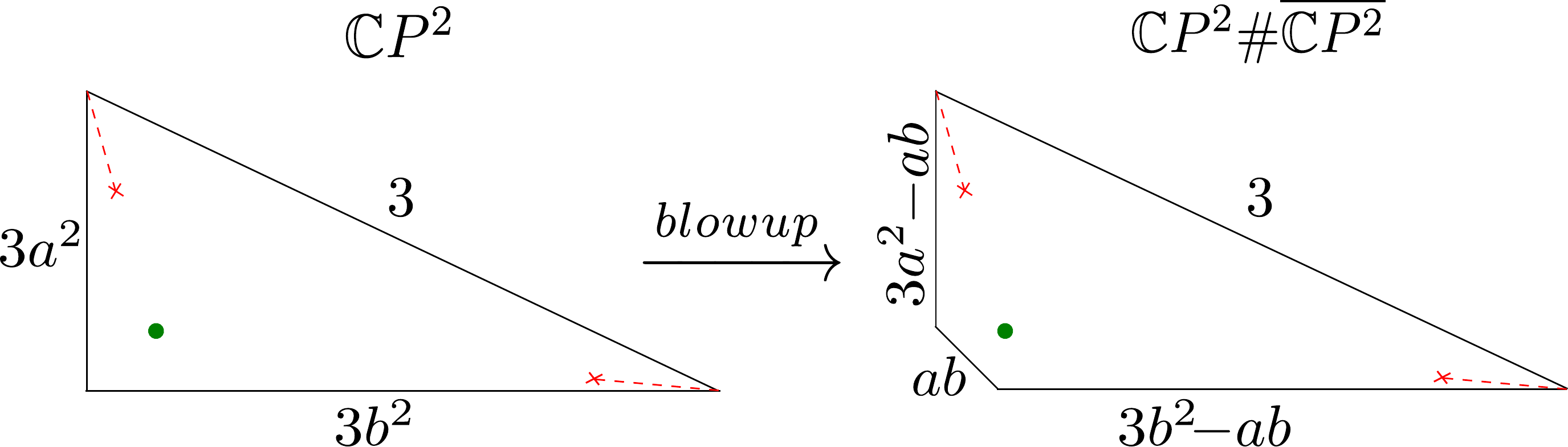}}

\caption{ATBD of $\BlI$ as a blowup of an ATBD of $\CP^2$ with affine 
lengths proportional to the squares of a Markov triple of the form $(1,a,b)$, with
$b > a$.}
\label{fig: ATBD1Blup} 

\end{center} 
\end{figure}

\begin{figure}[h!]   
  
\begin{center} 
  
\centerline{\includegraphics[scale=0.37]{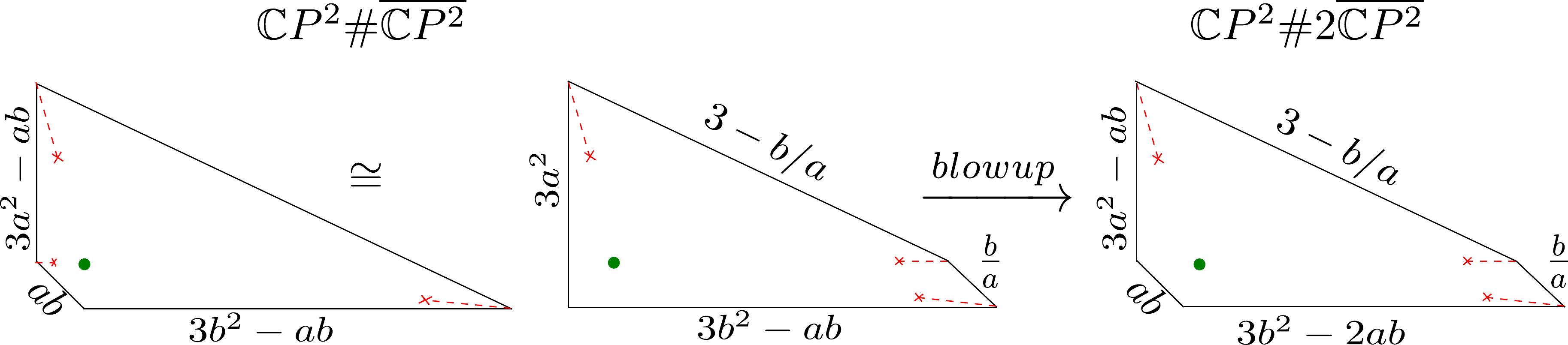}}

\caption{ATBD of $\BlII$ as a blowup of an ATBD of $\BlI$.}
\label{fig: ATBD2Blup} 

\end{center} 
\end{figure}

Eventhough we expect that each of these monotone Lagrangian mutually
belong to different symplectomorphism classes, for $k = 1 ,2$, we can't show that using
our technique. That is because we loose the positivity of intersection property
for the limit orbifold and hence we can't describe the boundary Maslov-2 convex hull
of $T_1(a,b)$ and $T_2(a,b)$.  

Nonetheless, for $k=1$, we can extract enough information about the boundary
Maslov-2 convex hull to show that there are infinitely many symplectomorphism classes
of monotone Lagrangian tori. More precisely, we can show that $\mho_{T_1(a,b)}$
must contain a vertex with {\it affine angle} $b' = 3a - b$ (the norm of
determinant of the matrix formed by the primitive vectors as columns). We can
also show that $\mho_{T_1(a,b)}$ is compact. Since we have infinitely many
possible values for $b'$, we must have infinitely many boundary Maslov-2 convex
hulls. Therefore, Theorem \ref{thm: main} \ref{thm: b} holds.
 
So we have:

\begin{cnj} \label{cnj: k2}
  There are infinitely many symplectomorphism classes of
  monotone Lagrangian tori inside $\BlII$.
\end{cnj}

Consider two monotone Lagrangian fibres of ATFs for which their ATBDs are related
via one mutation. The algebraic count of Maslov index 2 pseudo-holomorphic disks
for this tori is expected to vary according to wall-crossing formulas
\cite{Au07,Au09, GaUs10, Vi13}. In view of that we conjecture:
  
\begin{cnj} \label{cnj: ConvexHull} 
  
The boundary Maslov-2 convex hull of a monotone Lagrangian fiber of an ATF
described by an ATBD (whose cuts are inside eigenline of the respective node)
is determined by the limit orbifold. Actually, the vertices of the convex hull
should be the primitive vectors that describe the fan of the limit orbifold.
 
\end{cnj}

Which would allow us to conclude:  

\begin{cnj} \label{cnj: AllToriDifferent} 
  
Suppose we have two monotone Lagrangian fibres of ATFs of the same symplectic
manifold, described by ATBDs whose orbifold limits are different. Then they are
not symplectomorphic.
 
\end{cnj}

So we expect to have many more symplectomorphism classes of monotone Lagrangian tori
than the ones of $\Tpqr$.

Consider now a ATBD of triangular shape described in Figures \ref{fig:
triangshape1}-\ref{fig: triangshape4}. Call $X$ the corresponding del Pezzo
surface. Since they have no rank $0$ singularity, there is a smooth symplectic
torus $\Sigma$ living over the edge of the base of the corresponding ATF and
representing the anti-canonical class \cite{Sy03}. We can assume that a
neighbourhood of $\Sigma$ remains invariant under the mutations of the ATBD's,
so $\Sigma$ is always living over the edge of the base of corresponding ATF. 

Hence all the tori $\Theta^{n_1,n_2,n_3}_{p,q,r} \subset X$ live on $X \setminus 
\Sigma$ and are in different Hamiltonian isotopic classes there. The complement of (a 
neighbourhood) of $\Sigma$ has a contactype boundary $V$ with a Liouville vector
field pointing outiside. Hence we can atach the positive half of a symplectization,
obtaining $(X \setminus \Sigma)\cup(V \times [0,+\infty))$. We can show 

\begin{thm} \label{thm: X-Sigma}
  The tori $\Theta^{n_1,n_2,n_3}_{p,q,r} \subset (X \setminus \Sigma)\cup(V \times [0,+\infty))$
  belong to mutually different Hamiltonian isotopy classes.
\end{thm}

For the case of the complement of an elliptic curve in $\CP^2$ the above theorem 
can be proved by Tonkonog using a different approach. By looking at a 
Lagrangian skeleton of $X \setminus \Sigma$, Shende-Teumann-Williams \cite{STW15} can 
show that there exists infinitelly many distinct subcategories of the category
of microlocal sheaves on the Lagrangian skeleton. The Lagrangian skeleton is 
given by attaching Lagrangian disks to an torus. The subcategories mentioned 
above corresponds to sheaves on the tori given by mutations that are equivalent
to the ones we see in ATBDs, see Section \ref{subsec: STW}.

The rest of the paper is organised as follows:

We start defining some terminology in Section \ref{sec: term}. We suggest the
reader to move directly to Section \ref{sec: ATF} and use Section \ref{sec:
term} only if some terminology is not clear from the context. 

In Section \ref{sec: ATF}, we describe how to obtain all the ATBDs of Figures
\ref{fig: triangshape1}-\ref{fig: triangshape4}, also showing how to ``create
space'' to perform a blowup by changing the ATF. In Section \ref{subsec:
ATBlup}, we make a small digression to describe how to perform an almost toric
blowup. We believe that the reader should become easily acquainted with the
operations on the ATBDs and be able to deduce the moves just by looking 
at the Figures \ref{fig: 3Blup}-\ref{fig: 6Blup}, \ref{fig: 7Blup},\ref{fig: 
8Blup}. Nonetheless, we provide explicit description of each operation on the 
ATBDs.

In Section \ref{sec: Mutation}, we show that mutations of Markov type I
and II equations corresponds to mutations of ATBDs of triangular shape. 
We also show that any monotone ATBD of triangular shape is node-related
(Definition \ref{dfn: ATBDequation}) to a Markov type I equation. It 
follows from \cite[Section~3.5]{KaNo98} that Figures \ref{fig: 
triangshape1}-\ref{fig: triangshape4} provide a complete list of
ATBDs of triangular shape for del Pezzo surfaces.

In Section \ref{sec: Proof}, we compute the boundary Maslov-2 convex hull
$\mho_{\Theta^{n_1,n_2,n_3}_{p,q,r}}$ (Theorem \ref{thm: convex hull}),
which finishes the proof of Theorem \ref{thm: main}\ref{thm: a}. We 
prove Theorem \ref{thm: main}\ref{thm: b} in Section \ref{subsec: thmb}
and Theorem \ref{thm: X-Sigma} in Section \ref{subsec: X-Sigma}  

In Section \ref{sec: relating} we relate our work with \cite{STW15}, by pointing
out that the complement of the symplectic torus $\Sigma$ in the anti-canonical
class is obtained from attaching (Weinstein handles along the boundary of)
Lagrangian discs to the (co-disk bundle of the) monotone fiber of each ATBD. In
particular, these tori are exact on the complement of $\Sigma$. We also relate
our work with \cite{Ke15_1}, where Keating shows how modality 1 Milnor fibers
$\mathcal{T}_{p,q,r}$, for $(p,q,r) \in \{(3,3,3),(2,4,4),(2,3,6)\}$ compatify
to del Pezzo surfaces of degree $d= 3,2,1$. It follows from Theorem \ref{thm:
X-Sigma}, that there are infinitely many Hamiltonian isotopic classes of exact
tori in $\mathcal{T}_{p,q,r}$, for $(p,q,r) \in \{(3,3,3),(2,4,4),(2,3,6)\}$.
Also, in \cite[Section~7.4]{Ke15_2}, Keating mention that all Milnor fibers
$\mathcal{T}_{p,q,r}$ are obtained by attaching Lagrangian discs to a Lagrangian
torus as described in \cite{STW15}. We conjecture then that there are infinitely
many exact tori in $\mathcal{T}_{p,q,r}$. We believe this conjecture is also
made in \cite{STW15}. In Section \ref{subsec: KaNoHP}, we point out that the
Markov type I equations appear before related to $3$-blocks exceptional
collections in the del Pezzo surfaces \cite{KaNo98} and $\Q$-Gorenstein
smoothing of weighted projetive spaces to del Pezzo surfaces \cite{HaPr10}. We
enquire if there is a correspondence between ATBDs - $3$-blocks exceptional
collections - $\Q$-Gorenstein smoothings. Finally, we relate the ATBD of $\PxP$
in Figure \ref{fig: triangshape1} with the singular Lagrangian fibration given
by Fukaya-Oh-Ohta-Ono in \cite{FO312}, as well as a similar ATBD of $\CP^2$ with
the singular Lagrangian fibration described in \cite{Wu15}. In
\cite{FO312,FO311a} it was shown that there are a continuous of non-displaceable
fibers in the monotone $\PxP$ and in $\Blk$ for $k \ge 2$ with respect to some
symplectic form, but not monotone for $k > 2$. We ask what ATBDs have a
continuous of non-displaceable fibers.

\subsection*{Acknowledgements} We would like to thanks Ivan Smith, Dmitry
Tonkonog, Georgious Dimitroglou Rizell, Denis Auroux, Jonathan David Evans and
Kaoru Ono for useful conversations. Section \ref{sec: relating} was very much
motivated by talks given by Vivek Shende and Ailsa Keating at the Symplectic
Geometry and Topology Workshop held in Uppsala on September 2015.

\section{Terminology} \label{sec: term}

Before we describe how to get almost toric fibrations on all del Pezzo surfaces,
let's fix some terminology. A lot of the terminology can be intuitively grasped,
so we suggest the reader to move on to the next section and only use this 
section as a reference for terminology.

We recall that a \emph{primitive vector} on the standard lattice of $\R^2$ is an
integer vector that is not a positive multiple of another integer vector.

We also recall that an ATBD is the image of an affine map from the base of an 
ATF, minus a set of cuts, to $\R^2$ endowed with the standar affine structure.
Let $s$ be a node of an ATF and $R^+$ an eigenray leaving $s$. Suppose we have
an ATBD where the cut associated to $s$ is a ray equals to ``the image of'' $R^+$. 

\begin{dfn} \label{dfn: nodetype}
  We say that $R^+$ is an $(m,n)$-eigenray of an ATBD if it points towards the
  node in the direction of the primitive vector $(m,n) \in \Z^2 \subset \R^2$.
  We also say that $s$ is an $(m,n)$-node of the ATBD.
\end{dfn}
  
 We recall from \cite[Definition~2.1]{Vi14} that, a {\it transferring the cut}
 operation with respect to $R^+$ gives another ATBD, representing the same ATF,
 but with a cut (the image of) $R^-$, the eigenray other than $R^+$ pertaining
 to the same eigenline. In this paper we overlook the fact that we have two
 options (left and right) for performing a transferring the cut operation, since
 the two resulting ATBD are related via $SL(2,\Z)$.

\begin{dfn} \label{dfn: mutation}

We call a \emph{mutation} with respect to a \emph{$(m,n)$-node} an operation on
an ATBD containing a monotone fibre consisting of: a nodal slide
\cite[Section~6.1]{Sy03} of the corresponding $(m,n)$-eigenray passing through
the monotone fibre; and one transferring the cut operation with respect to the
same eigenray. 
\end{dfn}

\begin{dfn} \label{dfn: fullmutation}

A total \emph{mutation} is a mutation with respect to all $(m,n)$-nodes, for some
$(m,n)$. 
\end{dfn}

\begin{dfn}
  A Markov type I equation, is an integer equation for a triple $(p,q,r)$ of the
  form:
  
 \begin{equation} \label{eq: MarkovI}
 n_1p^2 + n_2q^2 + n_3r^2 = \sqrt{dn_1n_2n_3}pqr,
  \end{equation}
  
  for some constants $d,n_1,n_2,n_3 \in \Z_{>0}$, so that $dn_in_j \equiv 0 \mod
  n_k$, $\{i,j,k\} = \{1,2,3\}$ and $dn_1n_2n_3$ is a square. A solution
  $(p,q,r)$ is called a Markov type I triple, if $p,q,r \in \Z_{>0}$. \end{dfn}

\begin{dfn}

Let $(p,q,r)$ be a Markov type I triple. The Markov type I triple \\
$(p'= \sqrt{\frac{dn_2n_3}{n_1}}qr - p, q, r)$ is said to be obtained from $(p,q,r)$ via   
a mutation with respect to $p$. Analogous for mutation with respect to $q$ and 
$r$.  
   
\end{dfn}

\begin{dfn}
  A Markov type II equation, is an integer equation for a triple $(a,b,c)$ of the
  form:
  
  \begin{equation} \label{eq: Markovtype}
    k_1a^2 + k_2b^2 + k_3c^2 = Kk_1k_2k_3abc, 
  \end{equation}
  
  for some constants $K,k_1,k_2,k_3 \in \Z_{>0}$. A solution $(a,b,c)$ is called 
  a Markov type II triple, if $a,b,c \in \Z_{>0}$. 
  
\end{dfn}

\begin{dfn}

Let $(a,b,c)$ be a Markov type II triple. The Markov type II triple
$(a'= Kk_2k_3bc - a, b, c)$ is said to be obtained from $(a,b,c)$ via   
a mutation with respect to $a$. Analogous for mutation with respect to $b$ and 
$c$.

\end{dfn}

\begin{dfn}
 A Markov type I/II triple $(p,q,r)$/$(a,b,c)$ is said to be minimum if it minimizes
 the sum $p+q+r$/$a+b+c$, among Markov type I/II triples. 
\end{dfn}

\begin{dfn}
  An ATBD of triangular shape is an ATBD whose cuts are all in the direction of the 
 respective eigenrays of the associated node and whose closure is a triangle in $\R^2$.
\end{dfn}

\begin{dfn}
  An ATBD of \emph{length type} $(A,B,C)$, is an ATBD of triangular shape whose edges
  have affine lengths \emph{proportional} to $(A,B,C)$.
\end{dfn}

\begin{dfn}
  An ATBD of \emph{node type} $((n_1,p),(n_2,q),(n_3,r))$, is an ATBD of triangular 
  shape with the three cuts $R_1,R_2,R_3$ contains respectively $n_1,n_2,n_3$ 
  nodes, and the determinant of primitive vectors of the edges connecting at the cut 
  $R_1$, respectively $R_2$, $R_3$, have norm equals to $n_1p^2$, respectively,
  $n_2q^2$, $n_3r^2$.
\end{dfn}

Note that the above definition can be generalized to any ATBD whose cuts are
all in the direction of an eigenray leaving the respective node.

\begin{dfn} \label{dfn: ATBDequation} 

We say that an ATBD is length-related to an Markov type II equation \eqref{eq:
Markovtype} if it is of length type $(k_1a^2,k_2b^2,k_3c^2)$, for some Markov
type II triple $(a,b,c)$. \\
We say that an ATBD is node-related to an Markov type I
equation \eqref{eq: MarkovI} if it is of node type \\
$((n_1,p),(n_2,q),(n_3,r))$, for some Markov type I triple $(p,q,r)$, and
moreover the total space of the corresponding ATF is a del Pezzo of degree $d$,
i.e., for $\Blk$, $d = 9 - k$, and for $\PxP$, $d=2$. 

\end{dfn}

We also define the limit orbifold of an ATBD:

\begin{dfn} \label{dfn: limorb}
  Given an ATBD, its \emph{limit orbifold} is the orbifold for which the moment 
  map image is equal to the ATBD without the nodes and cuts, which are replaced 
  by corners (usually not smooth).
\end{dfn}

\section{Almost toric fibrations of del Pezzo surfaces} \label{sec: ATF}

To perform a blowup in the symplectic category \cite{MDSaBook_SympTop}, one
deletes a symplectic ball $B(\epsilon)$ of radius $\epsilon$ and collapses the
fibres of the Hopf fibration of $\del B(\epsilon)$ to points. In particular,
the blowup depends on the radius $\epsilon$ one takes. In a toric symplectic 
manifold, one can perform a blowup near a rank $0$ singularity and remain 
toric, provided one chooses an small enough ball compatible with the toric
fibration, see Figure \ref{fig: IntroBlup}.

\begin{figure}[h!]   
  
\begin{center} 
  
\centerline{\includegraphics[scale=0.4]{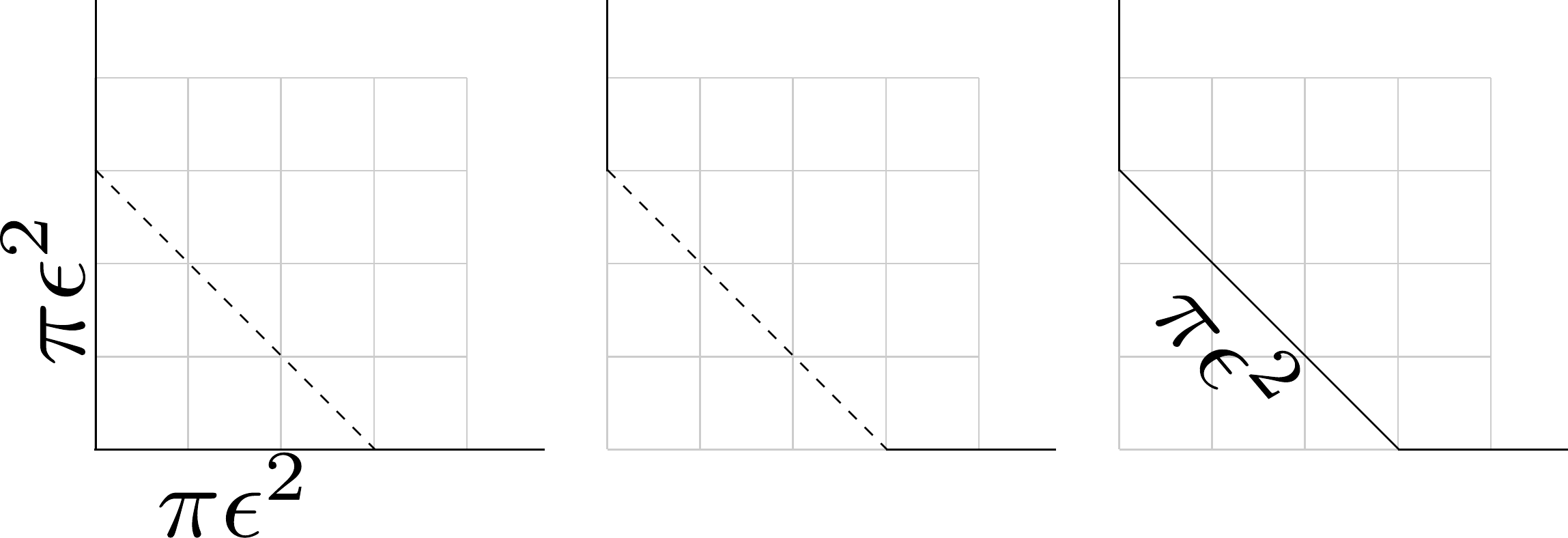}}

\caption{Blowup of a ball $B(\epsilon)$ centered at a rank 0 singularity in a toric manifold}
\label{fig: IntroBlup} 

\end{center} 
\end{figure}

  We recall that a symplectic manifold $(X, \omega)$ is said to be monotone if there exists 
  $C > 0$ such that $\forall H \in \pi_2(X)$:

\begin{equation}   
\int_H \omega = C \mathrm{c}_1(H).  \label{eq: mon}  
\end{equation}

  And Lagrangian $L \subset X$ is said to be monotone  if there exists 
  $C_L > 0$ such that $\forall \beta \in \pi_2(X, L)$:

\begin{equation}   
\int_\beta \omega = C_L\mu_L(\beta),  \label{eq: monLag}  
\end{equation}
where $\mu_L$ is the Maslov index.

Since $\mathrm{c}_1 = 2 \mu_{L|\pi_2(X)}$, if $\pi_2(X) \ne 0$, then $2C 
= C_L$. Also, if $L$ is orientable $\mu_L(\beta) \in 2\Z$.

The monotonicity condition is then affected by the size of the symplectic blow
up. In dimension 4, when we perform a symplectic blowup, we modify the second
homology group by adding a spherical class - coming from the quotient of
$S^3$ under the Hopf fibration - of Chern number 1. Therefore to keep
monotonicity one must choose the radius of the symplectic ball, so that the
quotient sphere has the appropriate symplectic area \eqref{eq: mon}.

One is able to perform symplectic blowup in one, two or three corners of the
moment polytope of $\CP^2$ to obtain monotone toric structures on $\BlI$,
$\BlII$, $\BlIII$. But one cannot go further, since it is not possible to
torically embed a ball of appropriate radius centered in a corner of the moment
polytope of $\BlIII$, depicted in the left-most picture of Figure \ref{fig:
3Blup}. 

Nonetheless, it is possible to create some space for the blowup if we only
require to remain almost toric. 

We are now ready to describe ATFs for all del Pezzo surfaces. In all ATBDs
apearing on Figures \ref{fig: 3Blup}, \ref{fig: 4Blup}, \ref{fig: 5Blup},
\ref{fig: 6Blup}, \ref{fig: 7Blup},\ref{fig: 8Blup},\ref{fig: PxP}, the interior
dot represents the monotone fibre. The reader should easily become familiar with
the operations and be able to read them from the pictures. Nonetheless, we give
explicit descriptions of the operations in each step.

\subsection{ATFs of $\BlIII$}

To arrive at an ATF of $\BlIV$, we perform some sequence of nodal trades and
mutations on the ATBDs of $\BlIII$ described on Figure \ref{fig: 3Blup}. And
eventually we are able to perform a blowup, and obtain an ATBD for $\BlIV$. We
are also albe to get the ATBD of triangular shape for $\BlIII$ (Figure \ref{fig: 3Blup}$(A_1)$) apearring in
Figure \ref{fig: triangshape1}.
The operations relating each diagram in Figure \ref{fig: 3Blup} are described below:

\begin{enumerate}[label=(\subscript{A}{\arabic*})]
    \item Toric moment polytope for $\BlIII$;
    \item Applied two nodal trades, getting $(1,0)$ and $(-1,0)$ nodes;    
    \item Mutated $(1,0)$-node and applied two nodal trades, getting $(0,1)$ and $(0,-1)$ nodes;     
    \item Mutated $(0,1)$-node and applied one nodal trade, getting a $(1,-1)$-node;
    \item Mutated $(1,-1)$-node.    
\end{enumerate}

 \begin{figure}[h!]   
  
\begin{center} 
  
\centerline{\includegraphics[scale=0.5]{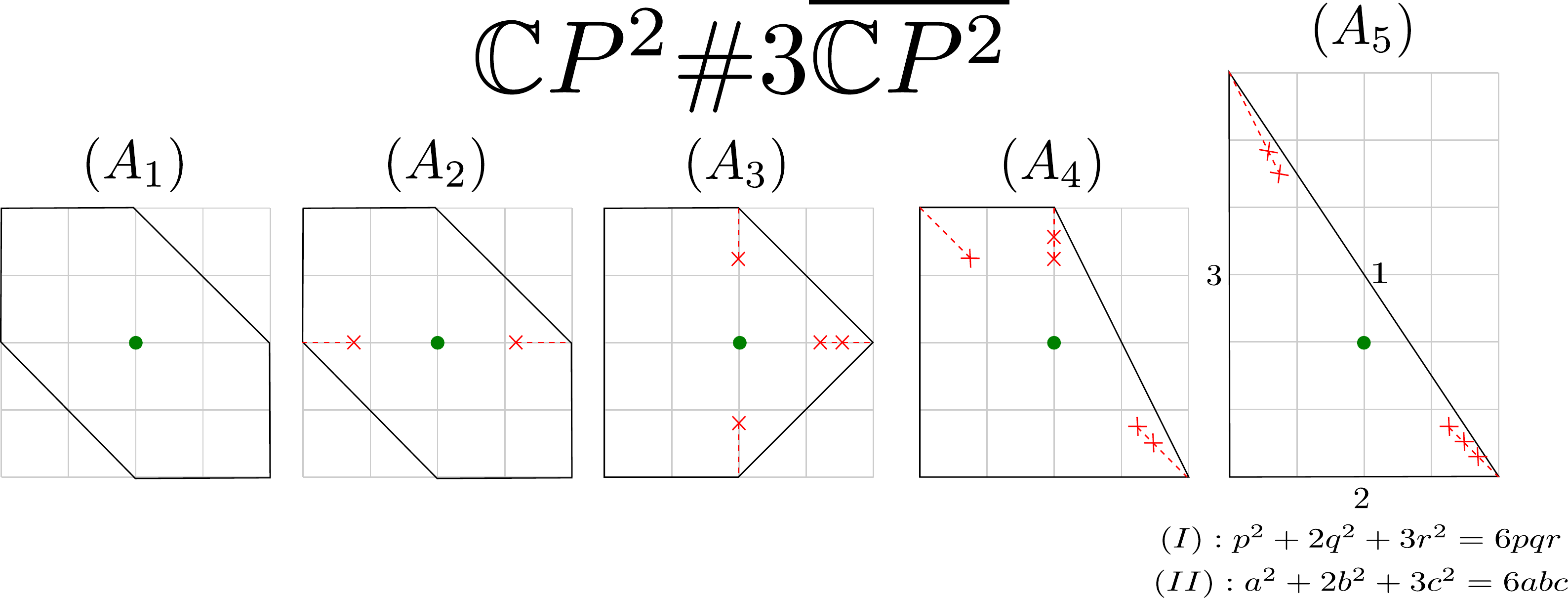}}

\caption{ATBDs of $\BlIII$.}
\label{fig: 3Blup} 

\end{center} 
\end{figure}

\subsection{ATFs of $\BlIV$}

We now see that we have created enough space to perform a toric blowup on the
corner (rank 0 singularity) of the 4th or 5th ATBD of Figure \ref{fig: 3Blup},
in order to obtain an ATF of $\BlIV$. We then perform some nodal trades and
mutations to, not only create more space for performing another blowup, but also
to get the ATBD of triangular shape in Figure \ref{fig: triangshape1}. 

The operations relating each diagram in Figure \ref{fig: 4Blup} are described below:

\begin{enumerate}[label=(\subscript{A}{\arabic*})]
    \item Blowup the corner of the ATBD $(A_5)$ of Figure \ref{fig: 3Blup};
    \item Applied one nodal trade, getting a $(0,1)$-node;         
    \item Mutated $(0,1)$-node.
    \item Mutated both $(1,-1)$-nodes.    
\end{enumerate}

\begin{figure}[h!]   
  
\begin{center} 
  
\centerline{\includegraphics[scale=0.45]{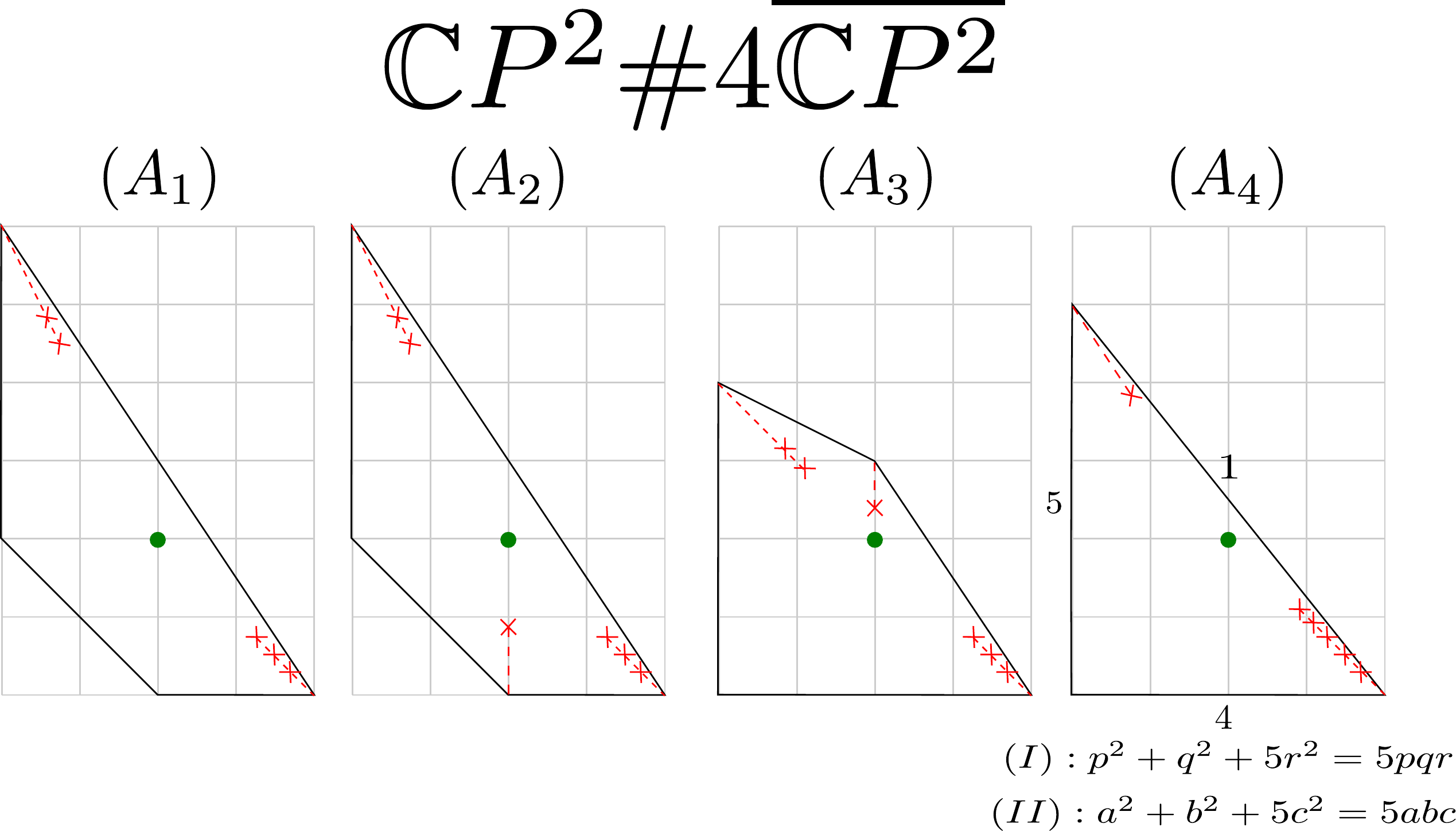}}

\caption{ATBDs of $\BlIV$. }
\label{fig: 4Blup} 

\end{center} 
\end{figure}

\subsection{ATFs of $\BlV$}
The ATDB of Figure \ref{fig:
triangshape2}, which is a $\pi/2$ rotation
of the ATBD $(B_2)$ in Figure \ref{fig: 5Blup}. The ATBD $(C_1)$ in Figure
\ref{fig: 5Blup} is used to perform another blowup.

\begin{figure}[h!]   
  
\begin{center} 
  
\centerline{\includegraphics[scale=0.365]{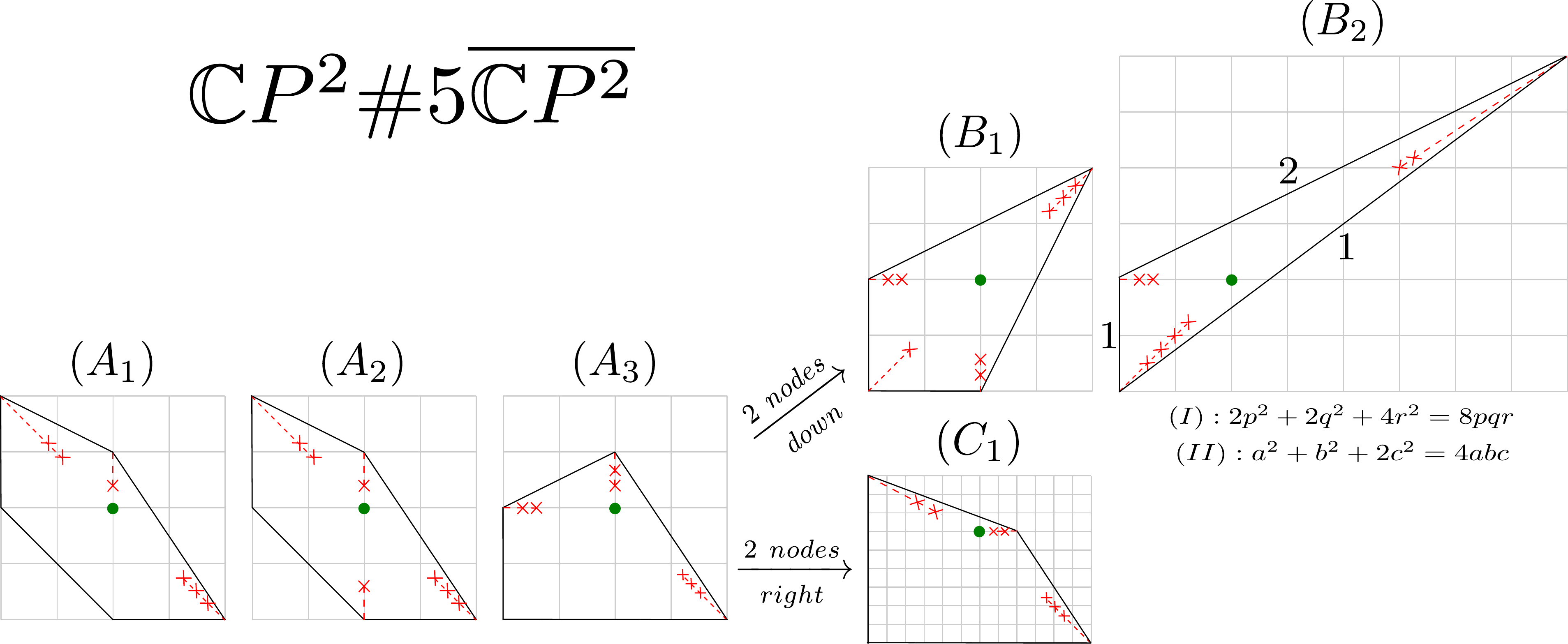}}

\caption{ATBDs of $\BlV$. }
\label{fig: 5Blup} 

\end{center} 
\end{figure}

The operations relating the $(A)$'s diagrams in Figure \ref{fig: 5Blup} are described below:

\begin{enumerate}[label=(\subscript{A}{\arabic*})]
    \item Blowup the corner of the ATBD $(A_3)$ of Figure \ref{fig: 4Blup};
    \item Applied one nodal trade, getting a $(0,1)$-node;         
    \item Mutated $(0,1)$-node.    
\end{enumerate}

Following the top arrow towards the $(B)$'s diagrams in Figure \ref{fig: 5Blup} 
we:

\begin{enumerate}[label=(\subscript{B}{\arabic*})]
    \item Mutated both $(0,1)$-nodes and applied one nodal trade, getting a $(1,1)$-node;         
    \item Mutated all three $(-1,-1)$-nodes.    
\end{enumerate}

To obtain $(C_1)$ ATDB we:

\begin{enumerate}[label=(\subscript{C}{\arabic*})]           
    \item Mutated both $(1,0)$-nodes.    
\end{enumerate}

\subsection{ATFs of $\BlVI$}

In Figure \ref{fig: 6Blup} we show how to get both ATBDs of Figure \ref{fig:
triangshape2}. We could have obtained an ATBD equivalent to the
ATBD $(B_2)$ directly from the ATBD $(A_3)$, but we will use the ATBD $(B_1)$
for blowup. We will also perform an almost toric blowup in the ATBD $(C_2)$.

The operations relating the $(A)$'s diagrams in Figure \ref{fig: 6Blup} are described below:

\begin{enumerate}[label=(\subscript{A}{\arabic*})]
    \item Blowup the corner of the ATBD $(C)$ of Figure \ref{fig: 5Blup};
    \item Applied one nodal trade, getting a $(1,0)$-node;         
    \item Mutated both $(-1,0)$-nodes.    
\end{enumerate}

Following the top arrow towards the $(B)$'s diagrams in Figure \ref{fig: 6Blup} 
we:

\begin{enumerate}[label=(\subscript{B}{\arabic*})]
    \item Mutated the $(0,1)$-node;         
    \item Mutated all three $(0,-1)$-nodes.    
\end{enumerate}

Following the bottom arrow from the $(A_3)$ ATBD towards the $(C)$'s diagrams in Figure \ref{fig: 6Blup} 
we:

\begin{enumerate}[label=(\subscript{C}{\arabic*})]
    \item Mutated only one $(0,-1)$-node;         
    \item Mutated all three $(-1,1)$-nodes.    
\end{enumerate}

\begin{figure}[h!]   
  
\begin{center} 
  
\centerline{\includegraphics[scale=0.4]{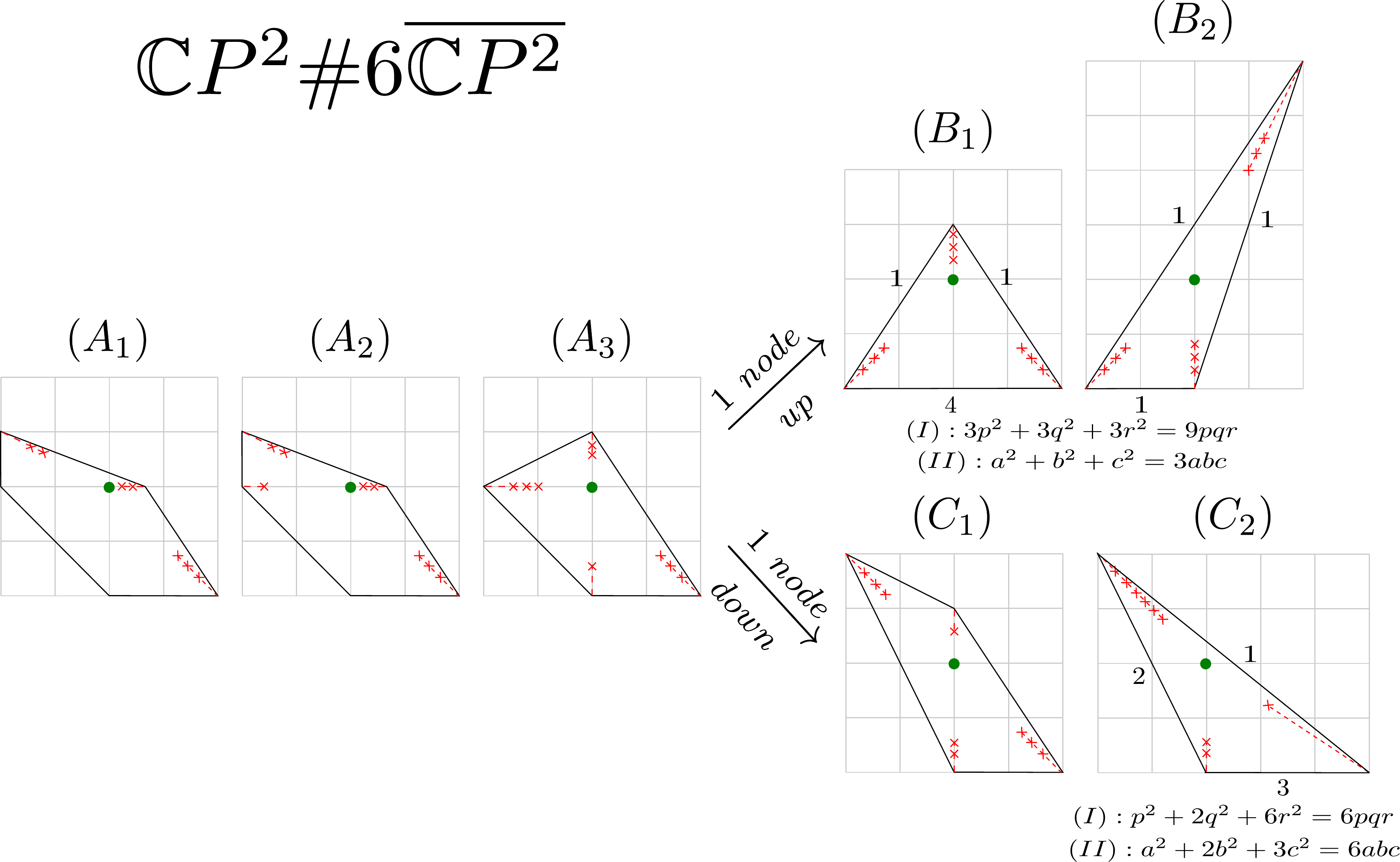}}

\caption{ATBDs of $\BlVI$. }
\label{fig: 6Blup} 

\end{center} 
\end{figure}

\subsection{Almost toric blowup} \label{subsec: ATBlup}

We will next perform a blowup now on a point in lying over a rank 1 elliptic
singularity of an ATF, corresponding to an edge of the ATBDs $(B_1)$ and $(C_2)$
in Figure \ref{fig: 6Blup}. That can be done while remaining in the almost toric
world.

\begin{figure}[h!]   
  
\begin{center} 
  
\centerline{\includegraphics[scale=0.35]{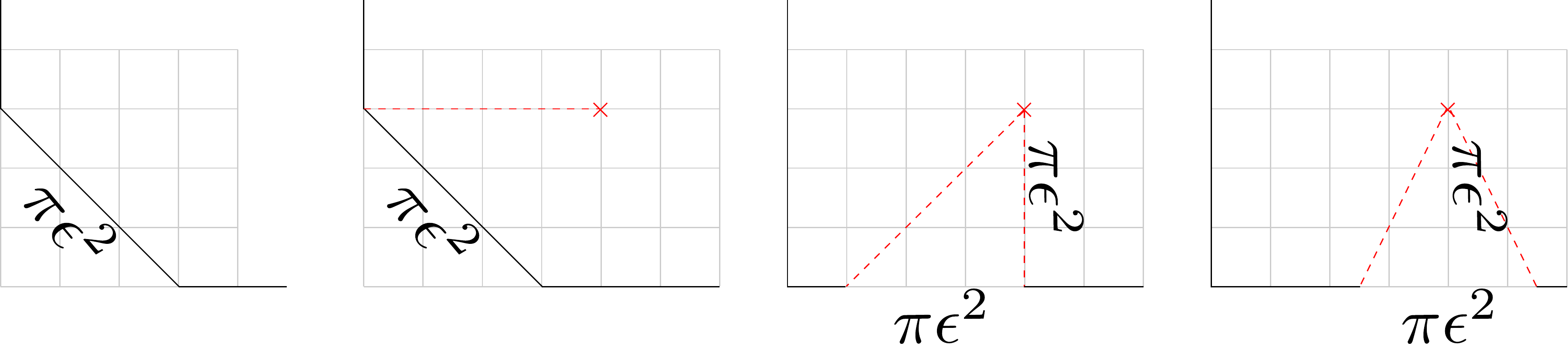}}

\caption{ATBDs on a toric blowup}
\label{fig: ATBlup0} 

\end{center} 
\end{figure}

We first point out that after a toric blowup (1st diagram of Figure \ref{fig:
ATBlup0}), we can get ATFs represented by the 2nd, 3rd and 4th ATBDs of Figure
\ref{fig: ATBlup0}. See also \cite[Figures 9,17]{Sy03}. Which makes us think
that we can get the 3rd and 4th ATBDs of Figure \ref{fig: ATBlup0}, by applying
a blowup on a point over the edge of the 1st ATBD. And indeed we can. In
\cite[Example 3.1.2]{Au09}, Auroux show how to construct an almost toric
fibration on the blowup of $\C^2$ over the point $(1,0)$, which lies on the
edge of the standard moment polytope of $\C^2$. We can than use this almost
toric fibration given on the neighborhood of the exceptional divisor as a local
model for what we call \emph{almost toric blowup}. The foollowing proposition
is an imediate consequence of \cite[Example 3.1.2]{Au09}.

\begin{prp}[\cite{Au09}(Example 3.1.2)] \label{prp: ATBlup0}
  
Consider the blowup at $(a,0) \subset \C^2$, with symplectic form
$\omega_\epsilon$, with respect to the standard ball of radius $\epsilon$. There
is an ATF on the blowup, with one nodal singularity whose monodromy's
eigendirection is parallel to the direction of the edge (image of the proper
transform of the $x$-axis). 
 
\end{prp}

The exeptional divisor lives over the segment $R$ on the base of the ATF
connecting the node to the edge, in direction ``orthogonal to the edge''. Let
$\mathcal{N}$ be the pre-image of a neighbourhood of the segment $R$. We can
identify the complement of $\mathcal{N}$ with the complement of the pre-image
of a similar neighbourhood on the toric moment polytope of $\C^2$, as in Figure
\ref{fig: ATBlup1}.

 \begin{figure}[h!]   
  
\begin{center} 
  
\centerline{\includegraphics[scale=0.35]{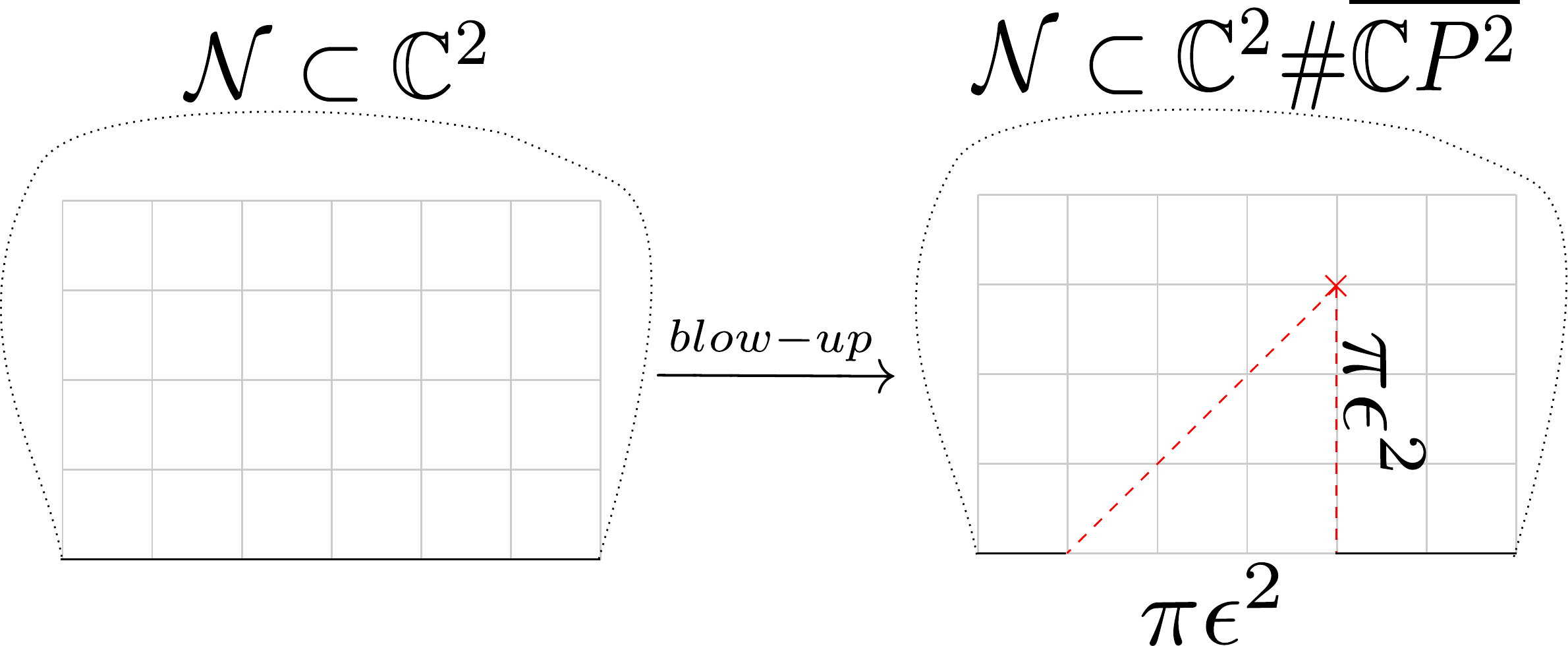}}

\caption{Almost toric blowup.}
\label{fig: ATBlup1} 

\end{center} 
\end{figure}

\begin{prp} \label{prp: ATBlup1}
  
Upon the above identification, the ATF of Proposition \ref{prp: ATBlup0} can be
made to agree with the toric one of $\C^2$ outside $\mathcal{N}$.

\end{prp}

\begin{proof} This follows from the fact that the symplectic form
$\omega_\epsilon$ agrees with the standard symplectic for of $\C^2$ outiside a
neighbourhood of the ball of radious $\epsilon$ centered at $(a,0)$. In that
region, the tori $L_{r,\lambda}$ described by Auroux in \cite[Example 3.1.2]{Au09}
coincides with the standard product torus in $\C^2$, i.e., outiside some
neighbourhood $\mathcal{N}$ as above, the fibres of the ATF
of the blowup are identified with the fibres of the standard moment polytope of
$\C^2$.
\end{proof}

Consider a rank 1 elliptic singularity $p$ of an ATF, lying over an edge of an
ATBD. Assume we have segment $R$ starting at (the image of) $p$ embedded inside
the ATBD (not crossing a cut or node), orthogonal to the edge containing $p$ and
of area $\pi \epsilon^2$. Also assume that the area of the edge to the left (or
right) of $R$ is bigger than $\pi\epsilon^2$ and let $q$ be a point on the edge
of distance $\pi\epsilon^2$ from (the image of) $p$. Let $S$ be the segment
uniting the end point of $R$ with $q$. Consider a symplectomorphism from an
almost toric neighborhood of $p$ to the toric neighborhood of a point $(a,0)
\subset \C^2$. From Propositions \ref{prp: ATBlup0}, \ref{prp: ATBlup1} we have:
  
\begin{prp} \label{prp: ATBlupFinal}
  There is an ATF on the $\epsilon$-blowup at the point $p$, with ATBD given 
  by replacing a neighbourhood of the base containing the segments $R$ and $S$,
  by an ATBD with cuts over $R$ and $S$ and a node in their intersection point. 
\end{prp}

\begin{dfn}
  We say that the ATBD on the blowup of Proposition \ref{prp: ATBlupFinal}, is 
  obtained from the previous one via a blowup of length $\pi \epsilon^2$.
\end{dfn}

\subsection{ATFs of $\BlVII$}

In the remaining sections, we refer to the length of an almost toric blowup in a
given ATBD according to the grid depicted. Recall that the invariant direction
of the monodromy is paralell to the edge containing the point we blowup.

To get the ATBDs $(A_1)$ and $(D_1)$ in Figure \ref{fig: 7Blup}, we apply an
almost toric blowup of length $2$ to the ATBDs $(B_1)$ and $(C_2)$ of Figure
\ref{fig: 6Blup}. Since $2$ is the distance from the monotone fiber to the
bottom edge, it is the area of an Maslov 2 disk lying over the vertical segment.
Because it is equal to the area of the exceptional curve, we remain monotone. 

After several mutations, we are able to get the ATBDs of Figure \ref{fig:
triangshape3}. As usual, we get ATBDs $(B_3)$ and $(C_2)$ to have space for the
next blowups.

The operations relating the $(A)$'s diagrams in Figure \ref{fig: 7Blup} are described below:

\begin{enumerate}[label=(\subscript{A}{\arabic*})]
    \item Applied an almost toric blowup of length 2 in the edge of the ATBD $(B_1)$ of 
    Figure \ref{fig: 6Blup};
    \item Transferred the cut towards the right edge, getting a $(-1,0)$-node;            
\end{enumerate}

Following the top arrow towards the $(B)$'s diagrams in Figure \ref{fig: 7Blup} 
we:

\begin{enumerate}[label=(\subscript{B}{\arabic*})]
    \item Mutated only one $(0,-1)$-node;         
    \item Mutated the $(-1,-1)$-node;
    \item Mutated all four $(0,1)$-nodes.    
\end{enumerate}

Following the bottom arrow from the $(A_2)$ ATBD towards the $(C)$'s diagrams in
Figure \ref{fig: 7Blup} we:

\begin{enumerate}[label=(\subscript{C}{\arabic*})]
    \item Mutated all three $(0,-1)$-nodes;         
    \item Mutated all six $(0,1)$-nodes.    
\end{enumerate}
 
Now we describe the operations relating the $(D)$'s diagrams in Figure \ref{fig: 7Blup}: 

\begin{enumerate}[label=(\subscript{D}{\arabic*})]
    \item Applied an almost toric blowup of length 2 in the edge of the ATBD $(C_2)$ of 
    Figure \ref{fig: 6Blup};
    \item Transferred the cut towards the left edge, getting a $(1,0)$-node;
    \item Mutated all six $(1,-1)$-nodes.            
\end{enumerate}
 
\begin{figure}[h!]   
  
\begin{center} 
  
\centerline{\includegraphics[scale=0.55]{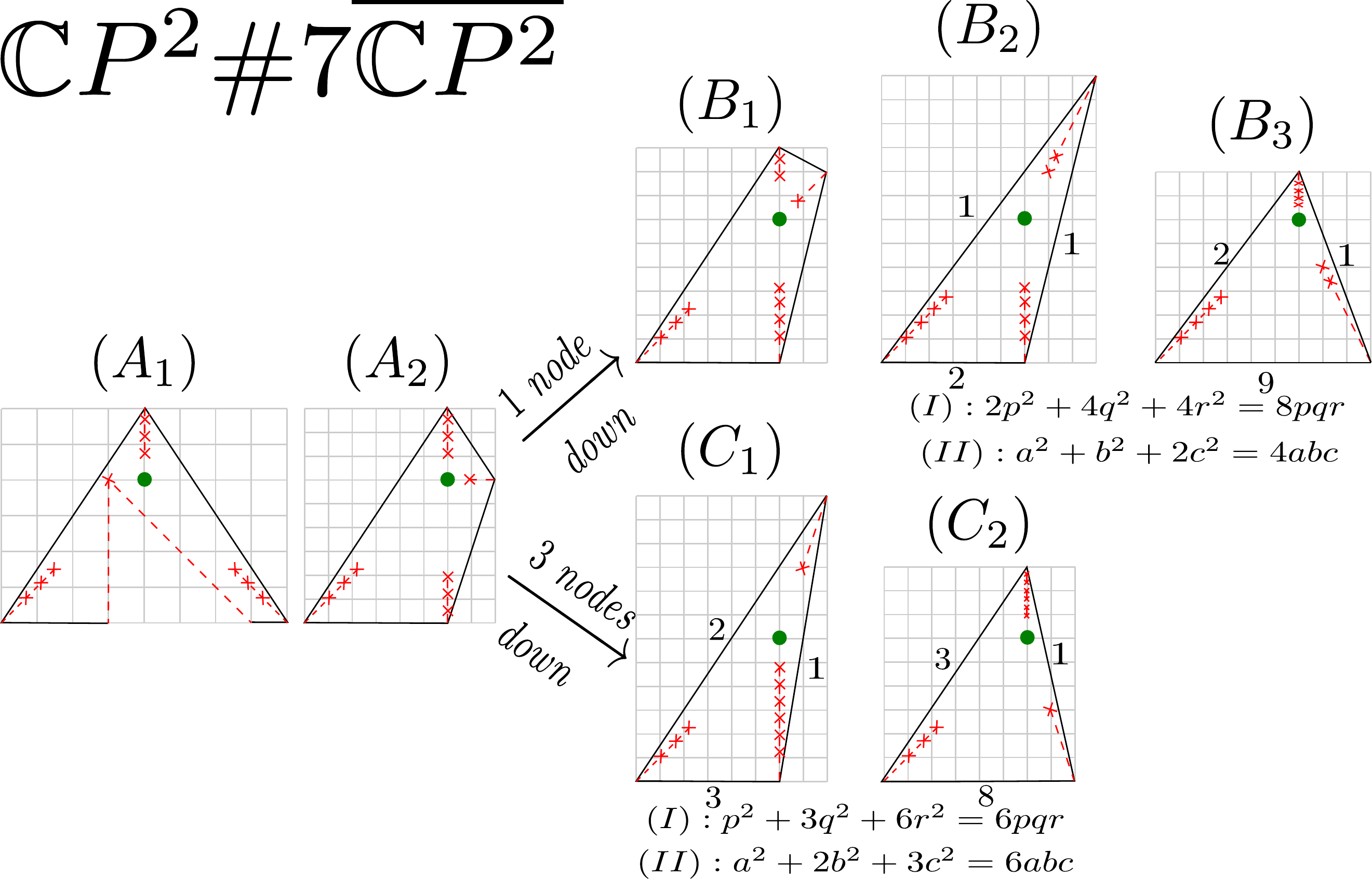}}
\vspace{0.5cm}
\centerline{\includegraphics[scale=0.4]{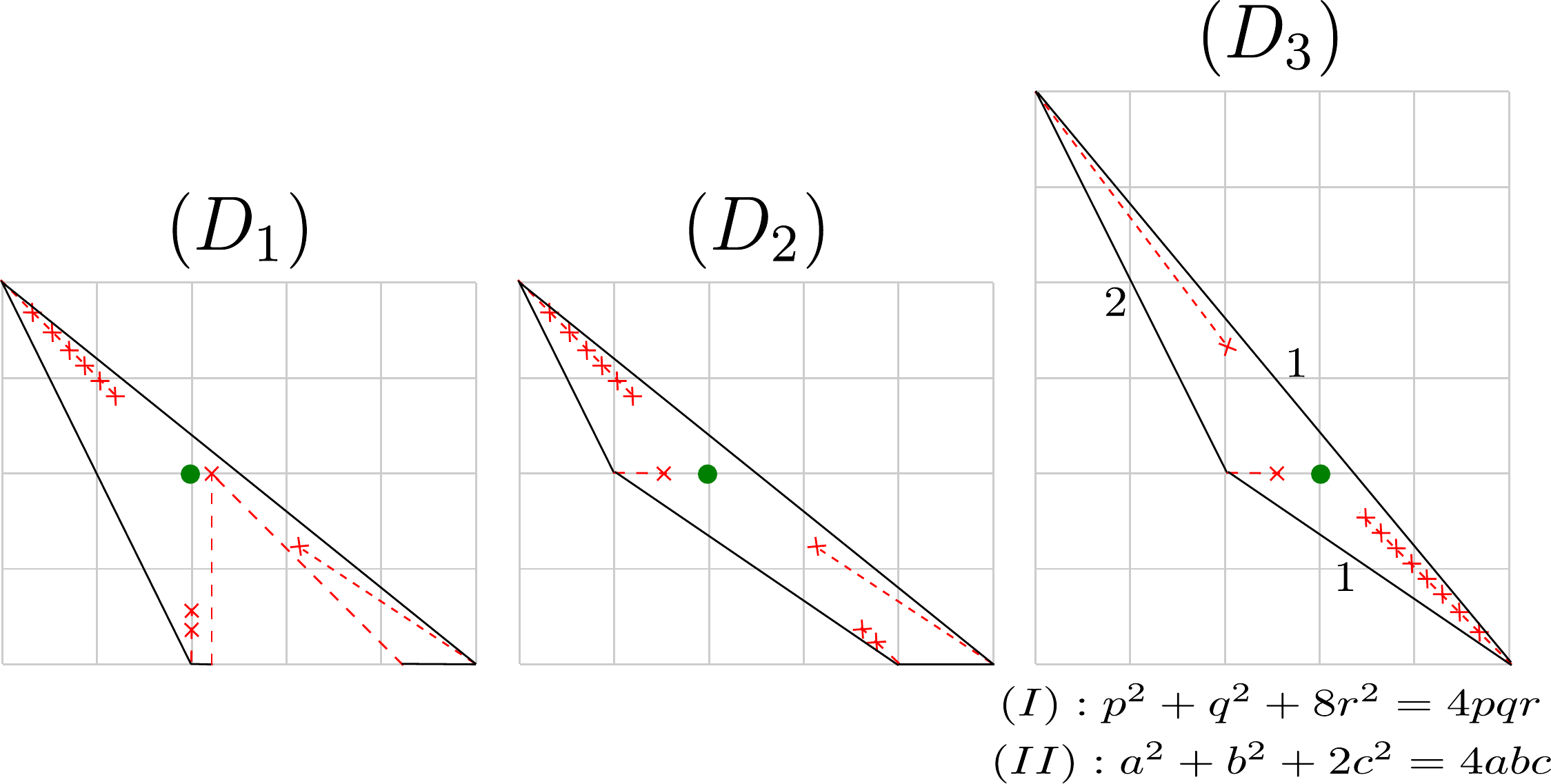}}
\caption{ATBDs of $\BlVII$. }
\label{fig: 7Blup} 

\end{center} 
\end{figure}

\newpage
\subsection{ATFs of $\BlVIII$}

We again blowup on edges of different ATBDs of $\BlVII$, namely $(B_3)$ and
$(C_2)$ in Figure \ref{fig: 7Blup}. Note that the blowup was not made in the
precise format given on Figure \ref{fig: ATBD2Blup}, but it is equivalent to an
honest one (up to nodal slide and change of direction of the cut). After
mutations we get the ATBDs of Figure \ref{fig: triangshape4}.

\begin{figure}[h!]   
  
\begin{center} 
  
\centerline{\includegraphics[scale=0.6]{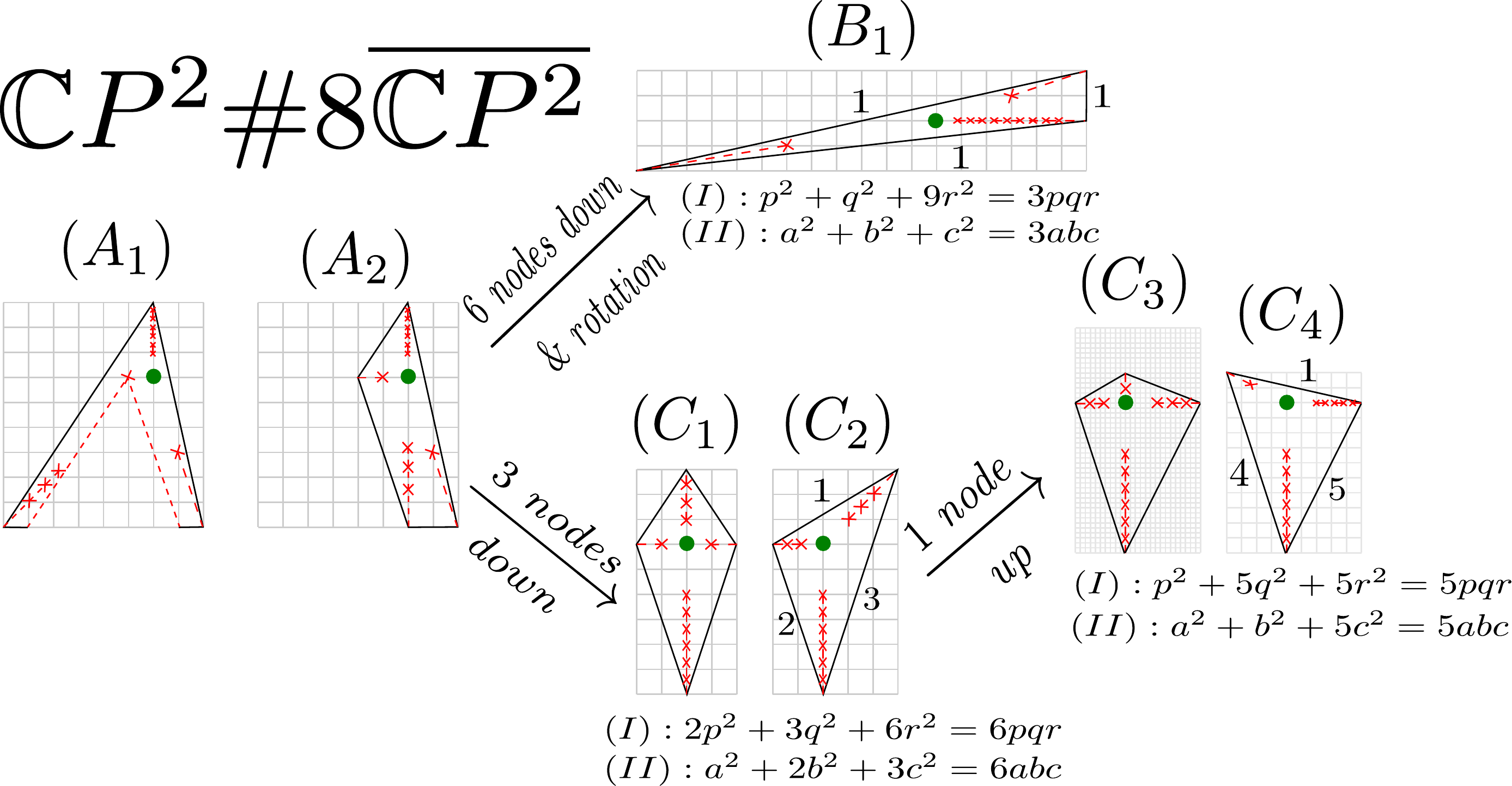}}
\vspace{0.5cm}
\centerline{\includegraphics[scale=0.53]{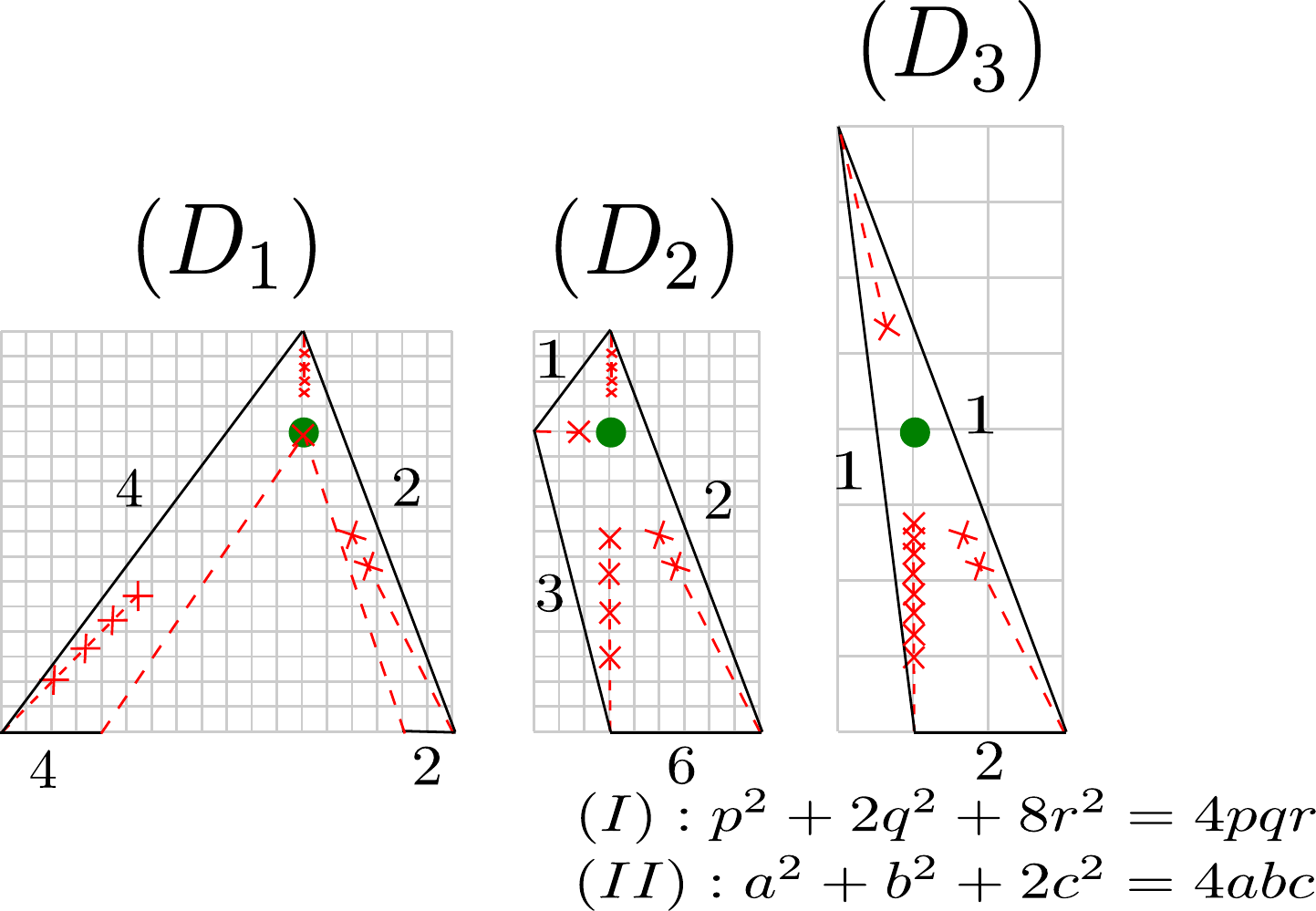}}

\caption{ATBDs of $\BlVIII$. }
\label{fig: 8Blup} 

\end{center} 
\end{figure}
   
The operations relating the $(A)$'s diagrams in Figure \ref{fig: 8Blup} are described below:

\begin{enumerate}[label=(\subscript{A}{\arabic*})]
    \item Applied an almost toric blowup of length 6 in the edge of the ATBD $(C_2)$ of 
    Figure \ref{fig: 7Blup};
    \item Transferred the cut towards the left edge, getting a $(1,0)$-node;            
\end{enumerate}

Following the top arrow towards the $(B)$'s diagrams in Figure \ref{fig: 8Blup} 
we:

\begin{enumerate}[label=(\subscript{B}{\arabic*})]
    \item Mutated all six $(0,-1)$-nodes and applied the
counter-clockwise $\pi/2$ rotation $(\in SL(2,\Z))$.             
\end{enumerate}

Following the bottom arrow from the $(A_2)$ ATBD towards the $(C)$'s diagrams in
Figure \ref{fig: 8Blup} we:

\begin{enumerate}[label=(\subscript{C}{\arabic*})]
    \item Mutated only three $(0,-1)$-nodes;         
    \item Mutated the $(-1,0)$-node;
    \item Mutated only one $(0,1)$-node;
     \item Mutated both $(1,0)$-nodes.    
\end{enumerate}
 
Finally, we describe the operations relating the $(D)$'s diagrams in Figure \ref{fig: 8Blup}: 

\begin{enumerate}[label=(\subscript{D}{\arabic*})]
    \item Applied an almost toric blowup of length 6 in the edge of the ATBD $(B_3)$ of 
    Figure \ref{fig: 6Blup} (the grid was refined so the blowup has length 12 on the new grid);
    \item Transferred the cut towards the left edge, getting a $(1,0)$-node;
    \item Mutated all four $(0,-1)$-nodes.            
\end{enumerate}



\subsection{ATFs of $\PxP$}
We finish by describing the ATBD of triangular shape for $\PxP$ appearring in
Figure \ref{fig: triangshape1}. Apply the counter-clockwise $\pi/2$ rotation
$(\in SL(2,\Z))$ to the ATBD $(A_3)$ of Figure \ref{fig: PxP} and get the ATBD
in Figure \ref{fig: triangshape1}. 

\begin{figure}[h!]   
  
\begin{center} 
  
\centerline{\includegraphics[scale=0.4]{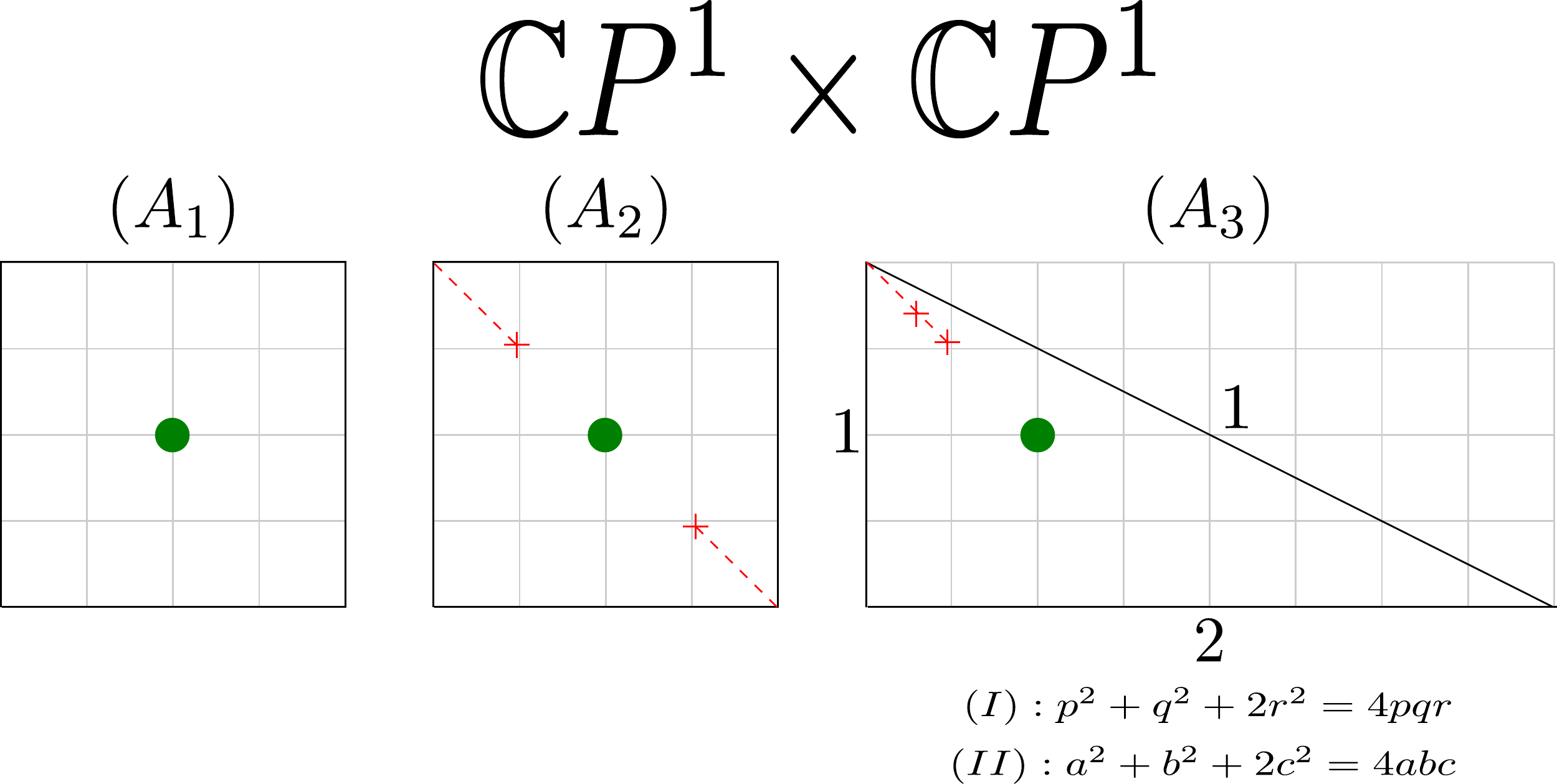}}

\caption{ATBDs of $\PxP$. }
\label{fig: PxP} 

\end{center} 
\end{figure}

\begin{enumerate}[label=(\subscript{A}{\arabic*})]
    \item Standard moment polytope of $\PxP$;
    \item Applied two nodal trades, getting a $(-1,1)$ and $(1,-1)$ nodes;
    \item Mutated the $(-1,1)$-node.            
\end{enumerate}

\section{Mutations and ATBDs of triangular shape} \label{sec: Mutation}

Let $\Delta$ be an ATBD of length type $(k_1a^2,k_2b^2,k_3c^2)$, where $(a,b,c)$
are Markov triples for equation \eqref{eq: Markovtype}. Assume that $\Delta$ has
a monotone fiber (not lying over a cut). Let $\bu_1, \bu_2, \bu_3$ be
primitive vectors in the direction of the edges of $\Delta$, so that $
k_1a^2\bu_1 + k_2 b^2\bu_2 + k_3 c^2\bu_3 = 0$. Up to $SL(2,\Z)$, we can assume
that $\bu_3 = (1,0)$. Let $\bw_i$ be the direction of the cut pointing towards 
the edge whose direction is $\bu_i$, see Figure \ref{fig: Delta}. Let $n_i$ be
the number of nodes in the cut $\bw_i$. Write $\bw_1 = (x,p)$ and $\bw_2 = (y,q)$.

\begin{figure}
\begin{center} 
  
\centerline{\includegraphics[scale=0.7]{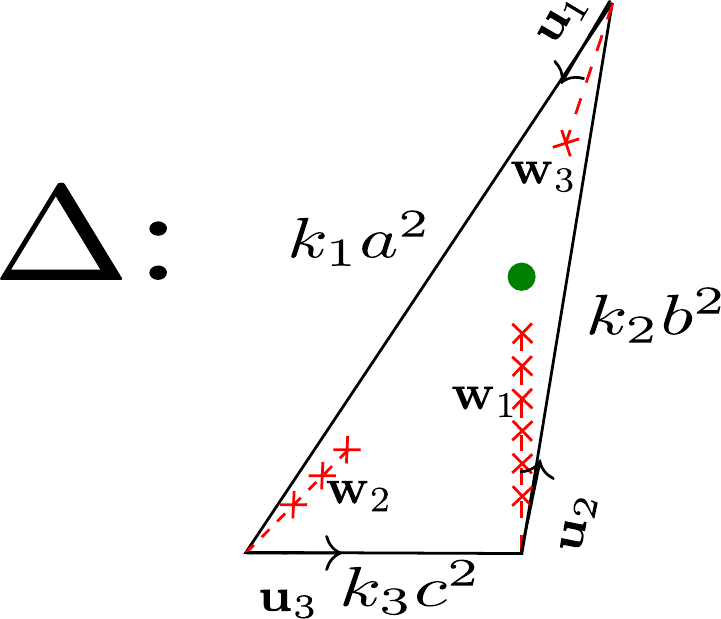}}
\caption{ATBD $\Delta$. }
\label{fig: Delta} 

\end{center} 
\end{figure}

The monodromy around a clockwise oriented loop surrounding $n$ nodes with
eigendirection $(s,t)$ is given by \cite[(4.11)]{Sy03}:

\begin{equation} \label{eq: monodromy}
  \begin{bmatrix} 1 - st & s^2 \\ -t^2 & 1+st \end{bmatrix}^n = 
\begin{bmatrix} 1 - nst & ns^2 \\ -nt^2 & 1 + nst  \end{bmatrix}
\end{equation}

So we have $\bu_1 = (1 - n_2yq, -n_2q^2)$ and $\bu_2 = (1 + n_1xp, n_1p^2)$. Note
that $n_1p^2 = |\bu_1 \wedge \bu_3|$ and $n_2q^2 = |\bu_3 \wedge \bu_2|$, which is
invariant under $SL(2,\Z)$. So, $|\bu_2 \wedge \bu_1| =  n_3 r^2$ for some $r \in \Z$.

\begin{prp} \label{prp: lambda}
  We have that:
  \begin{equation} \label{eq: lambda0}
    \frac{n_1p^2}{k_1a^2} = \frac{n_2q^2}{k_2b^2} = \frac{n_3r^2}{k_3c^2}
  \end{equation}
  We denote the above value by $\lambda$.
\end{prp} 

\begin{proof}
  Follows imediataly from $k_1a^2\bu_1 + k_2 b^2\bu_2 + k_3 c^2\bu_3 = 0$, that $\frac{n_1p^2}{k_1a^2} = 
  \frac{n_2q^2}{k_2b^2}$. Apply a $SL(2,\Z)$ map to conclude that it is also equal to $ 
  \frac{n_3r^2}{k_3c^2}$.
\end{proof}

It also follows from from $k_1a^2\bu_1 + k_2 b^2\bu_2 + k_3 c^2\bu_3 = 0$ and
the Markov type II equation that:

\begin{equation} \label{eq: lem1}
  Kk_1k_2k_3abc = k_1 a^2 + k_2b^2 + k_3c^2 = n_2yqk_1a^2 - n_1xpk_2b^2
  \end{equation}

It is worth noting that 

\begin{equation} \label{eq: lem00}
  k_1aa' = k_2b^2 + k_3c^2 
  \end{equation}

\begin{lem} \label{lem: Mutation}
 
 If we mutate all the $\bw_1$-nodes of $\Delta$ we obtain an ATBD of length type
 $(k_1a'^2,k_2b^2,k_3c^2)$, where $a' = Kk_2k_3bc - a$, i.e., $(a',b,c)$ is the
 mutation of $(a,b,c)$ with respect to $a$. Similarly, if we mutate all $\bw_2
 \backslash \bw_3$-nodes of $\Delta$, the length type changes according to the
 mutation of $(a,b,c)$ with respect to $b \backslash c$.

\end{lem}

\begin{proof} We need to prove that the affine length of the edges of the
mutated ATBD is proportinal to $(k_1a'^2,k_2b^2,k_3c^2)$. For convinience we
rescale the lengths of the edges of $\Delta$ by $a'$. 

The mutation glue the edges in directions $\bu_2$, $\bu_3$, to get an edge of
affine length $a'(k_3c^2 + k_2b^2) = ak_1a'^2$. As in Figure \ref{fig: LemMut},
assume we kept $\bw_2$ fixed and mutated $\bw_3$. The mutation divides the edge
in direction $\bu_1$ into two edges with affine lengths $\alpha$ and $a'k_1 a^2
- \alpha$. Say that the edge of length $a'k_1 a^2 - \alpha$ is opposite to the
$w_2$-eigenray. It is enough to show that $\alpha = ak_3c^2$, so $a'k_1 a^2 -
\alpha = a(aa'k_1 - k_3c^2) = ak_2b^2$.
   
We have that for some $\beta < 0$: 
\begin{equation} \label{eq: lem01}
 a'k_3c^2 \bu_3 = \beta \bw_1 - \alpha \bu_1
  \end{equation}
  
Hence

\begin{eqnarray}  
 0 = \beta p + \alpha n_2q^2 \\ \label{eq: lem02}
 a'k_3c^2 = \beta x + \alpha(n_2yq - 1)  \label{eq: lem03}
\end{eqnarray}
   
By equations \eqref{eq: lambda0} and \eqref{eq: lem02}, we have:    
   
\begin{equation} \label{eq: lem04}
\beta= - \alpha \frac{n_2 q^2}{p} = - \alpha\frac{k_2b^2n_1p}{k_1a^2}
  \end{equation}   

\newpage
Plugging into equation \eqref{eq: lem03} and using equation \eqref{eq: lem1}, we get:

\begin{eqnarray}  
a'k_3c^2 & = & \frac{\alpha}{a}\left[\frac{1}{k_1a}\left(-k_2b^2n_1px + k_1a^2n_2yq\right) - a \right] + \alpha n_2q^2  \nonumber \\
         & = & \frac{\alpha}{a}\left[Kk_2k_3bc -a\right] = a'\frac{\alpha}{a} \label{eq: lem05}
\end{eqnarray}

So indeed $\alpha = ak_3c^2$.
           
\end{proof}

A direct consequence of Proposition \ref{prp: lambda} and that $\Delta$ is
length-related to the Markov type II equation \eqref{eq: Markovtype} is

\begin{cor} \label{cor: MarkovI}
 
 Consider the ATBD $\Delta$. The numbers $n_1p^2 = |\bu_1 \wedge \bu_3|$,
 $n_2q^2 = |\bu_3 \wedge \bu_2|$, $n_3r^2 = |\bu_2 \wedge \bu_1|$ are so that
 $(p,q,r)$ is a Markov type I triple for the equation:
 
 \begin{equation*}
 n_1p^2 + n_2q^2 + n_3r^2 = \sqrt{dn_1n_2n_3}pqr,  \ \ \eqref{eq: MarkovI}
  \end{equation*}
 
where $d = \frac{K^2k_1k_2k_3}{\lambda}$. 
 
 \end{cor}

It follows then from the Lemma \ref{lem: Mutation} and Proposition \ref{prp: 
lambda}:

\begin{cor} \label{cor: MutationMarkovI}
 If we mutate all the $\bw_1$-nodes of $\Delta$ we obtain an ATBD of node type
 $((n_1,p'), (n_2,q), (n_3,r))$, where $(p',q,r)$ is a mutation of the $(p,q,r)$
 Markov type I triple. Analogously, for the other $\bw_i$-nodes.  
\end{cor}

We can verify for each ATBD in Figures \ref{fig: triangshape1}-\ref{fig: 
triangshape4}, that the value of $d$ in Corollary \ref{cor: MarkovI} is equal to
the degree of the corresponding del Pezzo. In fact one can prove that Corollary
\ref{cor: MarkovI}, without knowing beforhand that the ATBD is length-related
to a Markov type II equation \eqref{eq: Markovtype}.

\begin{thm} \label{thm: MarkovI}
  Let $\Delta'$ be an ATBD of node type
 $((n_1,p), (n_2,q), (n_3,r))$, such that the total space $X_{\Delta'}$ of the corresponding 
 ATF is monotone. Then $(p,q,r)$ is a Markov type I triple for \eqref{eq: 
 MarkovI}.
\end{thm}

\begin{proof} Let $(A,B,C)$ be the length type of $\Delta'$. We observe that in
Proposition \ref{prp: lambda} still holds, so we have

\begin{equation} \label{eq: lambda1}
\lambda = \frac{n_1p^2}{A} = \frac{n_2q^2}{B} = \frac{n_3r^2}{C}.
\end{equation} 

We will look at the self-intersection of the anti-canonical divisor class inside
the limit orbifold $X_o$, as defined in \cite{ChenRu02,Chen04}. The second
homology of the limit orbifold $X_o$ is one-dimensional, since the moment polytope
is a triangle. Denote by $H$ the generator of $H_2(X_o)$. 

\begin{clm} \label{clm: 1}
 $X_o$ is monotone. 
\end{clm}

\begin{proof}
 Consider a fiber $T$ and disks living over paths conecting $T$ to the edges in
the moment polytope of the limit orbifold $X_o$. These disks generate $H_2(X_o,T,
\Q)$, so some integer linear combination is a multiple of $H$ viewed in
$H_2(X_o,T,\Z)$. Therefore we can complete these disks with a 2-chain in $T$, to
get a cycle $m\hat{H}$ representing a multiple of $H$ lying away from the orbifold points.

Now, the complement of small neighbourhoods around the orbifold points can be
symplectically embedded into $X_{\Delta'}$, up to sliding the nodes close enough
to the edges, see \cite[Figure 7,Section 4.2]{Vi14}. Hence the Chern class and
symplectic area of $[m\hat{H}]$ coincide in both $X_{\Delta'}$ and $X_{o}$,
we get monotonicity for $X_o$.  
\end{proof}

\begin{clm} \label{clm: 2}
 The self-intersection of $H$ is
 
\begin{equation} \label{eq: H.H}
  H \cdot H = \frac{1}{\lambda ABC} = \frac{\lambda^2}{n_1n_2n_3p^2q^2r^2}.
\end{equation} 
 
\end{clm}

\begin{proof}
  The pre-image of the edges of length $A$ and $B$ intersect at one orbifold point and
  represent $AH, BH \in H_2(X_o)$. The degree of the orbifold point is the determinant
  of the primitive vectors of the edges, i~.e~., $n_3r^2$. Hence, $AH \cdot BH = 1/n_3r^2$
  \cite{ChenRu02,Chen04}, and the claim follows.  
\end{proof}

Now the anti-canonical divisor is represented by the pre-image of the three 
edges, and its self-intersection is the degree $d$ of $X_{\Delta'}$, by the same 
argument given in Claim \ref{clm: 1}. Therefore, by Claim \ref{clm: 2}:

\begin{equation} \label{eq: K.K}
 (A+B+C)H \cdot (A+B+C)H = \frac{(A+B+C)^2\lambda^2}{n_1n_2n_3p^2q^2r^2} = d.
\end{equation} 

Taking the square root and using \eqref{eq: lambda1}, we get the Markov type
I equation \eqref{eq: MarkovI}.

\end{proof}

From Theorem \ref{thm: MarkovI} and the work of Karpov-Nogin \cite{KaNo98} we
see that our list described in Figures \ref{fig: triangshape1}-\ref{fig: triangshape4}
describes all ATBDs of triangular shape.

\begin{prp}[Section 3.5 of \cite{KaNo98}] \label{prp: KaNo1}
  All Markov type I equations \eqref{eq: MarkovI} with $n_1 + n_2 + n_3 + d = 
  12$ are the ones appearing in Figures \ref{fig: triangshape1}, \ref{fig: 
  triangshape2}, \ref{fig: triangshape3}, \ref{fig: triangshape4}.
\end{prp}

And from the proposition below and Corollary \ref{cor: MutationMarkovI}, it
follows that each ATBD of Figures \ref{fig: triangshape1}-\ref{fig:
triangshape4}, gives rise to infinitely many ones.

\begin{prp}[Section 3.7 of \cite{KaNo98}] \label{prp: KaNo1}
  Any solution of the Markov type I equations appearing in Figures \ref{fig: triangshape1}, \ref{fig: 
  triangshape2}, \ref{fig: triangshape3}, \ref{fig: triangshape4}, can be 
  reduced to a minimum solution via a series of mutations. Moreover, for a 
  non-minimum solution 2 mutations increase the sum $p + q + r$ and one reduces.
\end{prp}

\section{Infinitely many tori} \label{sec: Proof}
\noindent We name $\Theta^{n_1,n_2,n_3}_{p,q,r}$ the monotone fiber of a monotone ATBD
of node-type $((n_1,p),$ $(n_2,q),(n_3,r))$. In this section we show that these 
tori live in mutually different symplectomorphism classes, completing the proof of
Theorem \ref{thm: main}\ref{thm: a}. We also show that there are infinitely many
symplectomorphism classes formed by the monotone tori $T_1(a,b)$ in $\BlI$, depicted in
Figure \ref{fig: ATBD1Blup}, proving \ref{thm: main}\ref{thm: b}. To finish the 
section we prove Theorem \ref{thm: X-Sigma}.

\subsection{ATBDs of triangular shape}

The Theorem \ref{thm: main}\ref{thm: a} follows from Theorem \ref{thm: convex hull} 
and the invariance of the boundary Maslov two convex hull for monotone 
Lagrangian \cite[Corollary~4.3]{Vi14}.

\begin{thm} \label{thm: convex hull}
  The boundary Maslov two convex hull of $\Theta^{n_1,n_2,n_3}_{p,q,r}$ is
  equal to the convex hull whose vertices are the primitive vectors generating the 
  fan of the corresponding limit orbifold. Moreover, the affine length of the 
  edges of $\mho_{\Theta^{n_1,n_2,n_3}_{p,q,r}}$ is $n_1p, n_2q, n_3r$.
\end{thm}

\begin{proof} The proof of the first part is totally analogous to the proof of
Theorem 1.1 \cite[Section~4]{Vi14}, so we will only sketch here. Denote by $M$
the del Pezzo surface and $\overline{M_+^\infty}$ the limit orbifold. 

We first consider contact submanifolds $V_i$'s bounding the symplectic
submanifols $M_{-i}$'s formed by the pre-image of an open sector of the ATBD
that encloses $n_i$ nodes, $i = 1,2,3$. We embed $M_{+} = M \setminus
(M_{-1}\cup M_{-2} \cup M_{-3})$ inside the limit orbifold. We pull back the
complex form from the orbifold to $M_+$ and extend it to $M$. We also pullback
the Maslov index 2 holomorphic discs $\alpha, \beta, \gamma$ 
\cite{ChoPo14}, which live in the complement of the orbifold points in the limit
orbifold. The boundary of their homology classes corresponds to the primitive
vectors generating the fan of $\overline{M_+^\infty}$, and they will give rise to the
vertices of $\mho_{\Theta^{n_1,n_2,n_3}_{p,q,r}}$. 

Given an pseudo-holomorphic disk $u$ with boundary on
$\Theta^{n_1,n_2,n_3}_{p,q,r}$, we look at its limit under neck-streching,
which is at the pseudo-holomorphic building \cite{CompSFT03, EliGiHo10}. The part of
the pseudo-holomorphic building lying in $M_+^\infty$, compactifyies to a degenerated
pdeudo-holomorphic disk $u^\infty_+$ in the limit orbifold $\overline{M_+^\infty}$, having 
the same boundary as $u$ upon identification of $\Theta^{n_1,n_2,n_3}_{p,q,r}$
as a fiber in the limit orbifold. 

We have that $[u^\infty_+]$ intersects positively each of the classes of the
pre-images $A,B,C$ of the edges of the limit orbifold. Indeed, consider a
component of $[u^\infty_+]$, either is not a multiple of $A$, and hence
intersects $A$ non-negatively \cite{ChenRu02,Chen04}, or it is a positive
multiple of $A$, which has positive self intersection, see Claim \ref{clm: 2}.
Similar for intersection with $B$ and $C$. 

Since each disk $\alpha, \beta$ or $\gamma$ intersects only one of the divisors 
$A,B,C$, and the plane of Maslov index 2 classes in the orbifold projects 
injectively to $H_1(\Theta^{n_1,n_2,n_3}_{p,q,r})$ under the boundary
map, we can conclude the first part of the Theorem.
 
For the second part, we use the notation in the description of $\Delta$ in
Section \ref{sec: Mutation}, see Figure \ref{fig: Delta}. The primitive vectors
for the fan of the limit orbifold are $\bv_3 = (0,1)$, $\bv_1 = (n_2q^2,
1-n_2yq)$, $\bv_2 = (n_1p^2, -n_1xp - 1)$, which are orthogonal to $\bu_3$,
$\bu_1$, $\bu_2$. Hence the affine length of the edges $\bv_1 - \bv_3 =
n_2q(q,-y)$ and $\bv_2 - \bv_3 = n_2p(p,-x)$ are respectively $n_2q$ and $n_1p$.
After applying a $SL(2,\Z)$ map, we can do the same analysis to conclude that
the affine length of the edge $\bv_1 - \bv_2$ is $n_3r$.
\end{proof}

\subsection{On $\BlI$} \label{subsec: thmb}

This section is devoted to prove Theorem \ref{thm: main}\ref{thm: b}. We start
with ATBDs of $\CP^2$ wiht two nodes and one corner (rank 0 elliptic
singularity) of length-type $(a^2,b^2,1)$, where $a^2 + b^2 + 1 = 3ab$. We
assume $a < b$, and scale the symplectic form so that the affine length of the
edges are $3$, $3a^2$, $3b^2$. We need to blow up so that the are of the
exceptional divisor $E$ is 1/3 of the area of the line. The area of the
anti-canonical divisor $3[\CP^1]$, represented by the preimage of the edges of
the ATBD \cite{Sy03}, is $3a^2 + 3b^2 + 1 = 9ab$. Hence we blow up so that
the symplectic area $\omega \cdot E = ab$. See Figure \ref{fig: ATBD1Blup}.
We note we have space to blowup, since $3a^2 - ab = ab' > 0$ and $3b^2 - ab = ba' >0$. 
We name the monotone torus $T_1(a,b)$.

We proceed as in the proof of Theorem \ref{thm: convex hull}, were we apply
neck-stretching ``degenerating'' the ATBD towards its limit orbifold. We will
use similar notation. 

Let's name $\alpha$, $\beta$, $\gamma$, $\varepsilon$ the classes of holomorphic
disks living in the complement of the orbifold points of the limit orbifold
$\overline{M^\infty_+}$. We consider a pseudo-holomorphic disk $u$ with
boundary on $T_1(a,b)$, and we look at the degenerate pseudo-holomorphic disk 
$u_+^\infty$ in the limit orbifold $\overline{M^\infty_+}$, which is the 
compactification of the top building of the neck-stretch limit.

We name $A$, $B$, $C$, the pre-image of the edges of the limit orbifold whose
symplectic area are respectively $3a^2 - ab$, $3b^2 - ab$, and $3$. We keep
calling $E$ the class of limit of the exceptional curve in the limit orbifold.
Say that $\alpha$ intesects $A$, $\beta$ intesects $B$, $\gamma$ intesects $C$
and $\varepsilon$ intesects $E$. 

\begin{prp}
  We have that $A \cdot A = \frac{a^2 - b^2}{b^2} < 0$, $B \cdot B = \frac{b^2 - a^2}{a^2} > 0$
  $C \cdot C = \frac{1}{a^2b^2} > 0$, $E \cdot E = -1 < 0$. 
\end{prp}

\begin{proof}
 Use the formula 
  $$ D\cdot D = \frac{|\bw \wedge \bv|}{|\bv \wedge \bu||\bw \wedge \bu|}, $$
 where $D$ is represented by an edge of the moment polytope of a tori orbifold with
 primitive vector $\bu$, and such that the primitive vectors of the adjacent edges are
 $\bv$ and $\bw$. We also have that $\bu$ points from the $\bv$-edge to the $\bw$-edge
 and $|\bv \wedge \bu| > 0$.
 
An alternative proof is to use that $A + E = a^2C$, $B + E = b^2 C$ and arguments similar
to Claim \ref{clm: 2}.  
\end{proof}

So the positivity of intersection argument given in the proof of Theorem \ref{thm: convex hull}
fails. Nonetheless we have:

\begin{lem} \label{lem: 1BlupIntersec}
 The intesrsections $E\cdot [u^\infty_+] \ge 0$, $B\cdot [u^\infty_+] \ge 0$ and 
 $C \cdot [u^\infty_+] \ge 0$. 
\end{lem}

\begin{proof}
  That $C \cdot [u^\infty] \ge 0$ and $B\cdot [u^\infty_+] \ge 0$ follows as before, since
  a component of $u^\infty_+$ contributing negatively to $C \cdot [u^\infty]$ would 
  have to be a positive multiple of $C$, but $C \cdot C > 0$. Similar for $B$ (in fact
  no component could be a multiple of $B$ by area reasons).
  
  That $E\cdot [u^\infty_+] \ge 0$ follows from $\omega^\infty_+ \cdot E =
  \omega^\infty_+ \cdot [u^\infty_+] = ab$. Since $u^\infty_+$ has the `main
  component' with boundary on (the limit of) $T_1(a,b)$, and all components have
  positive symplectic area, we can't have a multiple of $E$.  
     
\end{proof}

\begin{figure}[h]
\begin{center} 
  
\centerline{\includegraphics[scale=0.5]{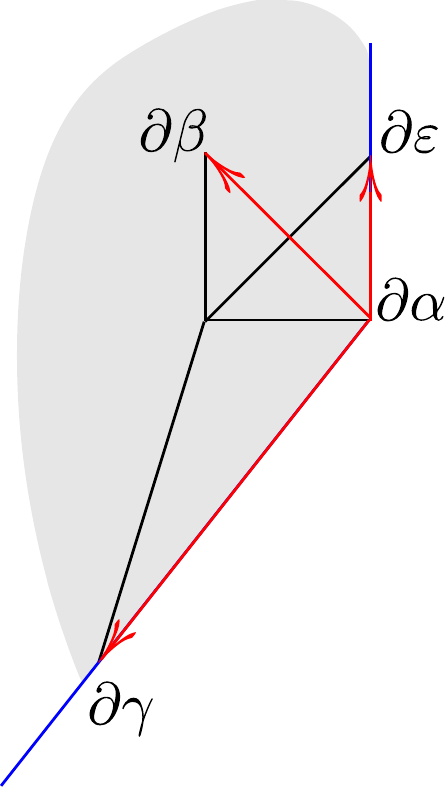}}
\caption{$\del \alpha$ is a corner of $\mho_{T_1(a,b)}$. }
\label{fig: Corner} 

\end{center} 
\end{figure}

\begin{lem} \label{lem: 1blupCorner}
  Upon identification of $T_1(a,b)$ with its limit in the limit orbifold, we have
  that $\del \alpha$ is a corner of the boundary Maslov-2 convex hull $\mho_{T_1(a,b)}$. 
\end{lem}

\begin{proof}

  First we notice that the count of pseudo-holomporphic disk in the class $\alpha$ is 
  $\pm 1$, by the same arguments as in \cite[Lemma~4.11]{Vi14}. 
  
  The classes of Maslov index 2 (or symplect area $ab$) disks with boundary in (the limit of)
  $T_1(a,b)$ inside the limit orbifold must be of the form:
  
\begin{equation} \label{eq: class}
 \alpha + k(\varepsilon - \alpha) + l(\beta - \alpha) + m(\gamma - \alpha)  
\end{equation}
  
  By Lemma \ref{lem: 1BlupIntersec}, we see that the class $[u^\infty_+]$ must
  have $k,l,m \in \Z_{\ge0}$ in \eqref{eq: class}. Hence $\del \alpha$ is
  a corner of $\mho_{T_1(a,b)}$, see Figure \ref{fig: Corner}.
  
\end{proof}

\begin{lem} \label{lem: affine angle}
  
  The affine angle of the corner $\del \alpha$ in $\mho_{T_1(a,b)}$, ~i.~e., the
  norm of the determinant of the primitive vectors of the edges of the corner, is $b' = 3a - b$.
   
\end{lem}

\begin{proof}
  We identify $\del \alpha = (1,0)$, $\del \beta = (0,1)$, $\del \epsilon = 
  (1,1)$, $(-a^2,-b^2)$. Recall that $1 + a^2 = bb'$. Hence 
  $\del \gamma - \del \alpha = -b(b',b)$. Therefore the primitive vectors
  are $(0,1)$ and $(b',b)$.
\end{proof}

\begin{lem} \label{lem: cptconvexhull}
  The Maslov-2 convex hull $\mho_{T_1(a,b)}$ is compact.
\end{lem}

\begin{proof}
  By area reasons, there is a constant $N_0 \in \Z_{>0}$ such that $[u_+\infty]$ cannot
  have $N_0$ or more components in the class $A$, for all pseudo-holomorphic disks 
  $u$. Therefore, there is a constant $N \in \Z_{>0}$ such that $A \cdot[u_+\infty] 
  > -N$. So, $k + l + m < N +1$ in the decomposition \eqref{eq: class} of 
  $[u_+\infty]$, which implies the Lemma.
\end{proof}

Recall that two Maslov-2 convex hull are equivalent if they are related via
$SL(2,\Z)$. Since we have an infinite number of values of affine angles, the
number of equivalence classes for the Maslov-2 convex hulls $\mho_{T_1(a,b)}$s
cannot be finite. This finishes the proof of Theorem \ref{thm: main}\ref{thm: 
b}.

\subsection{Proof of Theorem \ref{thm: X-Sigma}} \label{subsec: X-Sigma}

Let $(X, \omega)$ be a del Pezzo surface and consider $\Tpqr \subset X$ the monotone
Lagrangian tori described in this Section \ref{sec: Proof}. If one applies nodal
slide along a segment $[A,B]$ inside the base of an ATF, the new ATF can be
chosen to be equal to the previous one outside a small neighbourhood of $[A,B]$.
Therefore, we may assume that, provided we apply nodal trade for all nodes
beforhand, over the edge of all the ATFs desbribed in Section \ref{sec: ATF} for
the del Pezzo surface $X$ lives the same symplectic torus $\Sigma$. Also,
all the monotone Lagrangian tori $\Tpqr$ live in $X \setminus \Sigma$.  

Let $V = \del X$. In a neighbourhood of $V$, there is a Liouville vector field
pointing outside $V$ for which the corresponding Reeb vector field $R \subset
\Gamma(TV)$ is so that $\sfrac{V}{R} \cong \Sigma$. Therefore we may attach the
positive symplectization of $V$ to $X \setminus \Sigma$, and see $\Tpqr \subset
(X \setminus \Sigma)\cup(V \times [0,+\infty))$ as a fiber of an ATF. It follows
from seeing $X \setminus \Sigma$ coming from Weinstein handle attachments to the
co-disk bundle $D^*_{\epsilon}\Tpqr$ along the boundary of Lagrangian disks,
that $\Tpqr$ are exact in $(X \setminus \Sigma)\cup(V \times [0,+\infty))$, see
Section \ref{subsec: STW}.

\begin{thm*}[\ref{thm: X-Sigma}]
  The tori $\Theta^{n_1,n_2,n_3}_{p,q,r} \subset (X \setminus \Sigma)\cup(V \times [0,+\infty))$
  belong to mutually different Hamiltonian isotopy classes.
\end{thm*}
       
 \begin{proof}
  
  If there were a Hamiltonian isotopy between two of these tori, it could be
  made to be identity outside $(X \setminus \Sigma)\cup(V \times [0,C))$, for
  some constant $C$. Therefore it is enough to prove that the tori are not
  Hamiltonian isotopic in the del Pezzo surface $(X, \omega_C)$, obtained by
  collapsing the Reeb vector field inside $V \times \{C \}$. In other words, $(X,
  \omega_C)$ is obtained by inflating $X$ along $\Sigma$ by a factor of $C$. We
  note that $\Sigma \subset X$ not only represents the Poincar\'e dual to
  $\mathrm{c}_1$, but as a cycle in $X \setminus \Tpqr$, it represents the
  Poincar\'e dual to half of the Maslov class $\mu/2 \in H^2(X,\Tpqr)$. Looking
  at the ATF in $(X \setminus \Sigma)\cup(V \times [0,C)$ we see that, not only
  $\Tpqr \subset (X, \omega_C)$ remains monotone, it is the monotone fiber of an
  ATBD of $(X, \omega_C)$, which is the corresponding multiple of the initial
  ATBD of $(X, \omega)$. Hence the tori $\Tpqr$ are mutually non-Hamiltonian
  isotopic in $(X, \omega_C)$.

 \end{proof}
 
 \begin{rmk} Note that taking the complement of the divisor that is Poincar\'e
 dual to a multiple of the Maslov class for all tori was essential. All
 Lagrangian tori constructed in $\CP^2$ can be shown to live in the complement
 of a line. But Dimitroglou Rizzel's classification of tori in $\C^2$
 \cite{DR16} shows that there are only Clifford and Chekanov monotone Lagrangian
 tori in $\C^2$, up to Hamiltonian isotopy. \end{rmk}

\section{Relating to other works} \label{sec: relating}

\subsection{Shende-Treumann-Williams} \label{subsec: STW}

Consider a surface $S$ and a closed circle $\sigma \subset S$. Let $M = D^*S$
the co-disk bundle of $S$ (with respect to some auxiliary metric on $S$). We can
lift $\sigma$ to a Legendrian $\sigma' \subset \del M$, by considering over each
point $s \in \sigma$ the co-vector that vanishes on the tangent vectors of
$\sigma$ at $s$. We can then attach a Weinstein handle along $\sigma'$ obtaining
a new Liouville manifold $M_\sigma$. Shende-Treumann-Williams show that there is
a way to ``mutate'' the surface $S$ by sliding it along the Lagrangian core
$L_{\sigma'}$ of the Weinstein handle, obtaining a new exact Lagrangian
$S^\sigma$. We also note that $M_\sigma$ is homotopic equivalent to $S \cup
L_\sigma$, where $L_\sigma$ is a Lagrangian disk with boundary $\sigma \subset
S$ given by the continuation of the Lagrangian core $L_{\sigma'}$ where we
shrink the length of the co-vectors of $\sigma'$ to zero.


This idea can be generalized for any number of cycles $\sigma_1, \dots,
\sigma_n$ in $S$, where we obtain a Lioville manifold $M_{\sigma_1, \dots,
\sigma_n}$, with a Lagrangian skeleton given by $S \cup L_{\sigma_1}\cdots \cup
L_{\sigma_n}$, where $L_{\sigma_i}$ is a Lagrangian disk with boundary on
$\sigma_i$. Also, for any word $w$ on $\sigma_1, \dots, \sigma_n$, one can
obatain a surface $S^w$ mutated along Lagrangian disks according to the word
$w$. Moreover, Shende-Treumann-Williams show how to also mutate the Lagrangian
disks, to obtain Lagrangians disks $L^w_{\sigma^w_i}$'s with boundary on $S^w$ so
that $M_{\sigma_1, \dots, \sigma_n}$ has $S^w \cup L^w_{\sigma^w_1}\cdots \cup
L^w_{\sigma^w_n}$ as a Lagrangian skeleton. They also show how these mutations can
be encoded using cluster algebra.

\begin{qu}
  Do the $S^w$'s give infinitely many Hamiltonian isotopy classes
  in $M_{\sigma_1, \dots, \sigma_n}$ among exact Lagrangians?
\end{qu}

\begin{figure}[h]
\begin{center} 
  
\centerline{\includegraphics[scale=0.6]{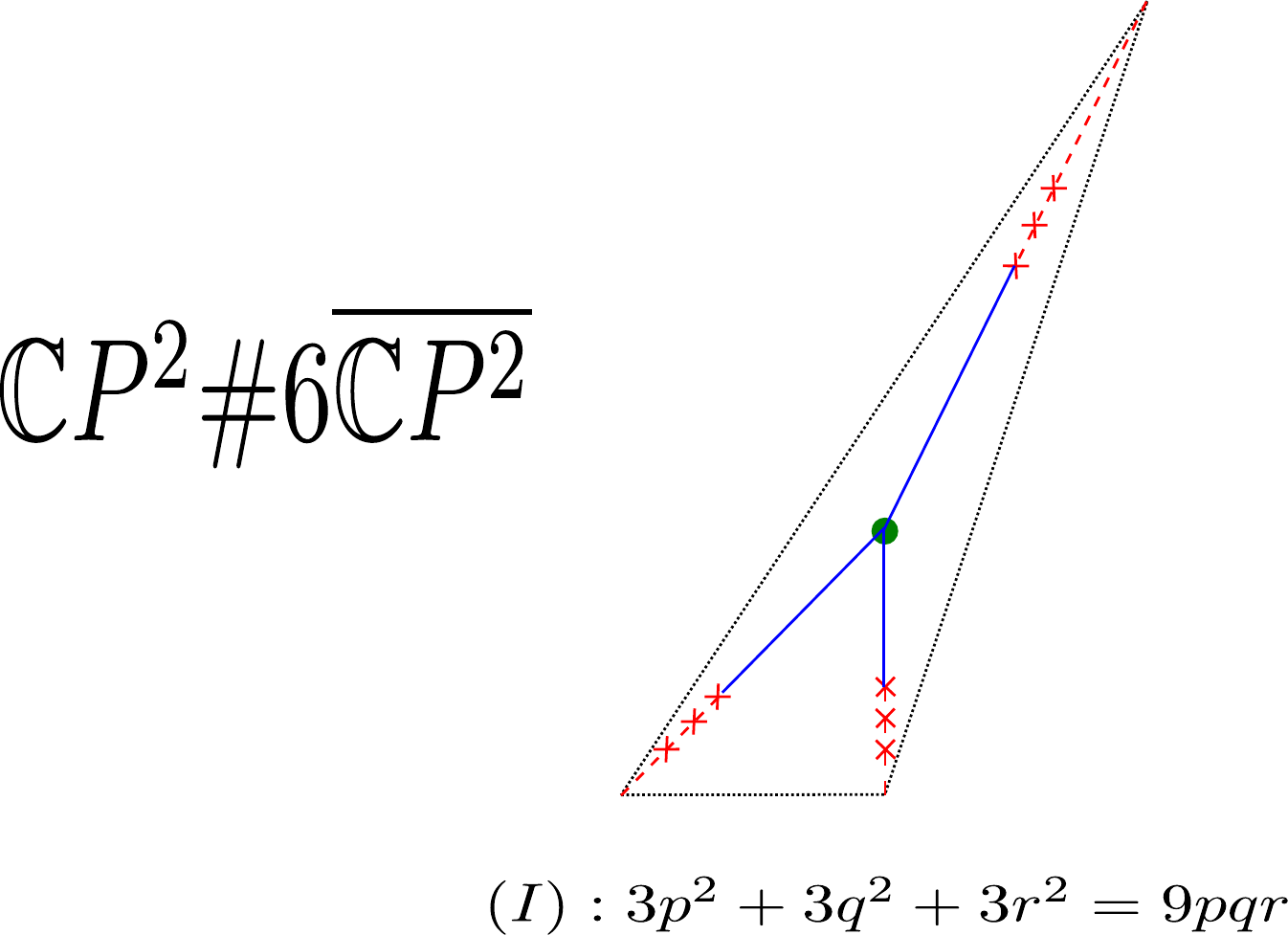}}
\caption{The complement of an anti-canonical surface in $\BlVI$ as a result of attaching nine
Lagrangian disks to the central fiber.}
\label{fig: 6BlupWeinHand} 

\end{center} 
\end{figure}

When $S$ is a torus $T$, this is precisely what is happening in the complement
of the anti-canonical surface $\Sigma$ living over the edge of the base of the
monotone ATFs. In fact in \cite{ToVi16}, we show that if the cycles $\sigma_i$'s
are taken to be periodic orbits of the self torus action of $T$, then we obtain
an ATF on $M_{\sigma_1, \dots, \sigma_n}$. Moreover, the ATF can be represented
by an ATBD whose cuts points towards the image of $T$ and are in the direction
of the cycles $\sigma$ via the identification $H_1(T,\R)$ with $\R^2 \supset$
ATBD. The Lagrangian disks live over the segment uniting $T$ with the cuts, see
Figure \ref{fig: 6BlupWeinHand}. In that way, the torus $T^w$ mutated with
respect to a word $w$ on $\sigma_1, \dots, \sigma_n$ corresponds to sliding the
nodes over the central fiber according to $w$. Also, if each time we slide a
node through the central fiber we perform a mutation on the ATBD, the mutated
Lagrangians $L^w_{\sigma^w_i}$'s can be read from the cuts of the mutated ATBD.

Figure \ref{fig: 6BlupWeinHand} illustrates $\BlVI \setminus \Sigma$, where we
have nine Lagrangian disks living over the segments connecting the central fiber
to each of the nodes of the ATBD. Theorem \ref{thm: X-Sigma} shows that total
mutations give rise to tori $\Theta^{3,3,3}_{p,q,r}$ living in mutually distinct
Hamiltonian isotopy classes.

\subsection{Keating}

In \cite[Proposition~5.21]{Ke15_1}, Keating shows how the Milnor fibers

\begin{eqnarray*} \label{eqn: Milnor Fibers}
 \mathcal{T}_{3,3,3} = \{x^3 + y^3 + z^3 + 1 = 0 \}; \\
 \mathcal{T}_{2,4,4} = \{x^2 + y^4 + z^4 + 1 = 0 \}; \\
 \mathcal{T}_{2,3,6} = \{x^2 + y^3 + z^6 + 1 = 0 \};  
\end{eqnarray*}
compactify respectivelly to the del Pezzo surfaces $\BlVI$, $\BlVII$, $\BlVIII$. 

Also, in \cite[Section~7.4]{Ke15_2}, Keating describes how the Milnor fibers
$\mathcal{T}_{p,q,r}$ of $x^p + y^q + z^r + axyz$ can be obtained by attaching
$p,q,r$ Weinstein handles to $D^*T$ along Legendrian lifts of three circles,
mutually intersecting at one point. By our discussion on the precious section,
we see the compactifications described in \cite[Proposition~5.21]{Ke15_1} depicted
in Figures \ref{fig: 6Blup}$(B_2)$, \ref{fig: 7Blup}$(B_2)$, \ref{fig: 8Blup}$(C_2)$.   

As a consequence of Theorem \ref{thm: X-Sigma}, we have that there are 
infinitely many exact Lagrangian tori in $\mathcal{T}_{3,3,3}, \mathcal{T}_{2,4,4},
 \mathcal{T}_{2,3,6}$.

\subsection{Karpov-Nogin and Hacking-Prokhorov } 
\label{subsec: KaNoHP}

There seems to be a relation between $m$-block collection of sheaves
\cite{KaNo98} on a del Pezzo surface and its ATBDs. By the Theorem in
\cite[Section~3]{KaNo98}, a $3$-block collection on a degree $d$ del Pezzo
suface containing exceptional collections with $n_1$,$n_2$,$n_3$ sheaves of
ranks $p,q,r$, satisfy the Markov type one equation \eqref{eq: MarkovI}. 

\begin{qu}
 
 Suppose we have an ATBD with node type $((n_1,p_1), \cdots, (n_m,p_m))$ of a del
 Pezzo surface. Is there an $m$-block collection $(\mathcal{E}_1, \cdots,
 \mathcal{E}_m)$, so that the exeptional colection $\mathcal{E}_i$ contains
 $n_i$ sheaves of rank $p_i$?
   
\end{qu}

There is also a relation between $\Q$-Gorenstein smoothing \cite{HaPr10} and
ATBD on the smooth surface. In \cite[Theorem~1.2]{HaPr10}, it is shown that the
weighted projective planes that admits a $\Q$-Gorenstein smoothing are precisely
the ones given by the limit orbifolds of the ATBDs obtianed by total mutation of
the ones in Figure \ref{fig: triangshape1}. It seems that pushing the nodes 
towards the edges of an ATBD corresponds to degenerating the del Pezzo surface
to the limit orbifold.

\begin{qu}
  Do monotone Lagrangians of a del Pezzo know about its degenerations? In other 
  words, do the limit orbifold of an ATBD has a $\Q$-Gorenstein smoothing to
  the corresponding del Pezzo?
\end{qu}

\subsection{FOOO and Wu}

The singular fibration of $\PxP$ described in \cite{FO312} can be thought as a
degeneration of the ATBD $(A_3)$ of Figure \ref{fig: PxP} where both nodes
aproach the edge, but instead of degenerating to an orbifold point, a Lagrangian
sphere that lives between the nodes survive, see \cite[Remark~3.1]{ToVi15}.
Similarly, the ATBD of $\CP^2$ with limit orbifold $\CP(1,1,4)$, can be thought
to degenerate to the singular fibration described in \cite{Wu15}, where instead
of an orbifold point we have a Lagrangian $\R P^2$. 

A way to see this degenerations is using the auxiliary Lefschetz fibration which
compactifies $(x,y) \in \C^2 \mapsto xy \in \C$ to $\PxP \backslash \CP^2$ to
get ATFs, as in \cite{Au07,Au09}. A (possibly singular) Lagrangian torus is the
union of orbits of the $S^1$-action $e^{i\phi}\cdot(x,y) = (e^{i\phi}x,
e^{-i\phi}y)$ with the same moment image, living over a circle in the base. A
Lagrangian $S^2 \backslash \R P^2$ lives over the segment $\R_{\ge 0}\cup \{
\infty\}$ and is formed by orbits of the $S^1$-action with moment image $0$.
Consider a continuous family of foliations $\mathcal{F}_t$, $t \in
[0,\epsilon)$, of the base $\CP^1$ so that: for $t \ne 0$, $\mathcal{F}_t$ is
formed by circles degenerating to the points $-1$ and $1$; and $\mathcal{F}_0$ is
formed by circles degenerating to the point $-1$ and the segment $\R_{\ge 0}\cup
\{ \infty\}$. The ATF of $\PxP \backslash \CP^2$ is formed by Lagrangian tori
living over some folliation $\mathcal{F}_\delta$ for a small $\delta$. The
singular fibration described in \cite{FO312} $\backslash$ \cite{Wu15} can be
interpreted as a singular Lagrangian fibration, where each Lagrantian is given
by the orbits of the $S^1$-action with moment image $\mu$ living over each leave
of the folliation $\mathcal{F}_\mu$, if $\mu \le \delta$, and living ove the
fibration $\mathcal{F}_\delta$ for $\mu \ge \delta$. So, for $\mu = 0$, we
get a family of Lagrangian tori, projecting to the circles of $\mathcal{F}_0$,
degenerating to $S^1$ at one end and to a Lagrangian $S^2 \backslash \R P^2$ at
the other.

The point is: in \cite{FO312} it is shown that there is a continuum of
non-displaceable fibers of the singular fibration described. It follows that
there is a continuum of non-displaceable fibers on the ATBD $(A_3)$ of Figure
\ref{fig: PxP}. The non-displaceability of the analougous fibers for the $\CP^2$
case is an open question. A more general question is:

\begin{qu}
  Which of the ATBDs described in this paper have a continuum of 
  non-displaceable fibers?
\end{qu}

\bibliographystyle{alphanum}
\bibliography{SympRefs}

\newcommand{\etalchar}[1]{$^{#1}$}
\begin{thebibliography}{FOOO3}

\bibitem[Aur1]{Au07}
D.~Auroux.
\newblock {Mirror symmetry and T-duality in the complement of an anticanonical
  divisor}.
\newblock {\em J. G\"okova Geom. Topol.}, 1:51--91, 2007.

\bibitem[Aur2]{Au09}
D.~Auroux.
\newblock Special {L}agrangian fibrations, wall-crossing, and mirror symmetry.
\newblock In {\em Surveys in differential geometry. {V}ol. {XIII}. {G}eometry,
  analysis, and algebraic geometry: forty years of the {J}ournal of
  {D}ifferential {G}eometry}, volume~13 of {\em Surv. Differ. Geom.}, pages
  1--47. Int. Press, Somerville, MA, 2009.

\bibitem[BEH{\etalchar{+}}]{CompSFT03}
F.~Bourgeois, Y.~Eliashberg, H.~Hofer, K.~Wysocki, and E.~Zehnder.
\newblock {Compactness results in Symplectic Field Theory}.
\newblock {\em Geom. Topol.}, 7:799--888, 2003.

\bibitem[Che]{Chen04}
W.~Chen.
\newblock Orbifold adjunction formula and symplectic cobordisms between lens
  spaces.
\newblock {\em Geom. Topol.}, 8:701--734 (electronic), 2004.

\bibitem[CP]{ChoPo14}
H.~Cho, C and M.~Poddar.
\newblock Holomorphic orbi-discs and {L}agrangian {F}loer cohomology of
  symplectic toric orbifolds.
\newblock {\em J. Differential Geom.}, 98(1):21--116, 2014.

\bibitem[CR]{ChenRu02}
W.~Chen and Y.~Ruan.
\newblock Orbifold {G}romov-{W}itten theory.
\newblock In {\em Orbifolds in mathematics and physics ({M}adison, {WI},
  2001)}, volume 310 of {\em Contemp. Math.}, pages 25--85. Amer. Math. Soc.,
  Providence, RI, 2002.

\bibitem[DR]{DR16}
G.~Dimitroglou~Rizell.
\newblock {The Classification of Lagrangian tori in a four-dimensional
  symplectic vectorspace}.
\newblock {\em In preparation}.

\bibitem[EGH]{EliGiHo10}
Y.~Eliashberg, A.~Givental, and H.~Hofer.
\newblock {Introduction to Symplectic Field Theory}.
\newblock In {\em {Visions in Mathematics}}, pages 560--673. Birkh{\"a}user,
  2010.

\bibitem[EP]{ElPo93}
Yakov Eliashberg and Leonid Polterovich.
\newblock Unknottedness of {L}agrangian surfaces in symplectic {$4$}-manifolds.
\newblock {\em Internat. Math. Res. Notices}, (11):295--301, 1993.

\bibitem[FOOO1]{FO3Book}
K.~Fukaya, Y.-G. Oh, H.~Ohta, and K.~Ono.
\newblock {\em {Lagrangian Intersection Floer Theory: Anomaly and
  Obstruction}}, volume~46 of {\em {Stud. Adv. Math.}}
\newblock American Mathematical Society, International Press, 2010.

\bibitem[FOOO2]{FO311a}
K.~Fukaya, Y.-G. Oh, H.~Ohta, and K.~Ono.
\newblock {Lagrangian Floer theory on compact toric manifolds II: bulk
  deformations}.
\newblock {\em Selecta Math. (N.S.)}, 17(2):609--711, 2011.

\bibitem[FOOO3]{FO312}
K.~Fukaya, Y.-G. Oh, H.~Ohta, and K.~Ono.
\newblock {Toric Degeneration and Nondisplaceable Lagrangian Tori in $S^2\times
  S^2$}.
\newblock {\em Internat. Math. Res. Notices}, 13:2942–2993, 2012.

\bibitem[Gro]{Gr85}
M.~Gromov.
\newblock {Pseudo holomorphic curves in symplectic manifolds}.
\newblock {\em Invent. Math.}, 82:307--347, 1985.

\bibitem[GU]{GaUs10}
S.~Galkin and A.~Usnich.
\newblock Laurent phenomenon for landau-ginzburg potential.
\newblock available at
  \texttt{http://research.ipmu.jp/ipmu/sysimg/ipmu/417.pdf}, 2010.

\bibitem[HP]{HaPr10}
P.~Hacking and Y.~Prokhorov.
\newblock Smoothable del {P}ezzo surfaces with quotient singularities.
\newblock {\em Compos. Math.}, 146(1):169--192, 2010.

\bibitem[Kea1]{Ke15_2}
A.~Keating.
\newblock {Homological mirror symmetry for hypersurface cusp singularities}.
\newblock {\em arXiv:1510.08911}, 2015.

\bibitem[Kea2]{Ke15_1}
A.~Keating.
\newblock Lagrangian tori in four-dimensional {M}ilnor fibres.
\newblock {\em Geom. Funct. Anal.}, 25(6):1822--1901, 2015.

\bibitem[KN]{KaNo98}
B.~V. Karpov and D.~Y. Nogin.
\newblock Three-block exceptional sets on del {P}ezzo surfaces.
\newblock {\em Izv. Ross. Akad. Nauk Ser. Mat.}, 62(3):3--38, 1998.

\bibitem[Laz]{La00}
L.~Lazzarini.
\newblock Existence of a somewhere injective pseudo-holomorphic disc.
\newblock {\em Geom. Funct. Anal.}, 10(4):829--862, 2000.

\bibitem[LS]{SyLe10}
N.~C. Leung and M.~Symington.
\newblock Almost toric symplectic four-manifolds.
\newblock {\em J. Symplectic Geom.}, 8(2):143--187, 2010.

\bibitem[McD1]{MD90}
D.~McDuff.
\newblock The structure of rational and ruled symplectic {$4$}-manifolds.
\newblock {\em J. Amer. Math. Soc.}, 3(3):679--712, 1990.

\bibitem[McD2]{MD96}
Dusa McDuff.
\newblock From symplectic deformation to isotopy.
\newblock In {\em Topics in symplectic {$4$}-manifolds ({I}rvine, {CA}, 1996)},
  First Int. Press Lect. Ser., I, pages 85--99. Int. Press, Cambridge, MA,
  1998.

\bibitem[MS]{MDSaBook_SympTop}
D.~McDuff and D.~A. Salamon.
\newblock {\em {Introduction to symplectic topology}}.
\newblock {Oxford Mathematical Monographs}. 1998.

\bibitem[OO1]{OhtaOno96}
H.~Ohta and K.~Ono.
\newblock Notes on symplectic {$4$}-manifolds with {$b^+_2=1$}. {II}.
\newblock {\em Internat. J. Math.}, 7(6):755--770, 1996.

\bibitem[OO2]{OhtaOno97}
H.~Ohta and K.~Ono.
\newblock Symplectic {$4$}-manifolds with {$b^+_2=1$}.
\newblock In {\em Geometry and physics ({A}arhus, 1995)}, volume 184 of {\em
  Lecture Notes in Pure and Appl. Math.}, pages 237--244. Dekker, New York,
  1997.

\bibitem[Sal]{Sa13Survey}
D.~Salamon.
\newblock Uniqueness of symplectic structures.
\newblock {\em Acta Math. Vietnam.}, 38(1):123--144, 2013.

\bibitem[STW]{STW15}
V.~Shende, D.~Treumann, and H.~Williams.
\newblock {Cluster varieties and Symplectic 4-manifolds}.
\newblock {\em In preparation}.

\bibitem[Sym]{Sy03}
M.~Symington.
\newblock Four dimensions from two in symplectic topology.
\newblock In {\em Topology and geometry of manifolds ({A}thens, {GA}, 2001)},
  volume~71 of {\em Proc. Sympos. Pure Math.}, pages 153--208. Amer. Math.
  Soc., Providence, RI, 2003.

\bibitem[Tau1]{Ta95}
C.~H. Taubes.
\newblock The {S}eiberg-{W}itten and {G}romov invariants.
\newblock {\em Math. Res. Lett.}, 2(2):221--238, 1995.

\bibitem[Tau2]{Ta96}
C.~H. Taubes.
\newblock {${\rm SW}\Rightarrow{\rm Gr}$}: from the {S}eiberg-{W}itten
  equations to pseudo-holomorphic curves.
\newblock {\em J. Amer. Math. Soc.}, 9(3):845--918, 1996.

\bibitem[Tau3]{Ta00Book}
C.~H. Taubes.
\newblock {\em Seiberg {W}itten and {G}romov invariants for symplectic
  {$4$}-manifolds}, volume~2 of {\em First International Press Lecture Series}.
\newblock International Press, Somerville, MA, 2000.
\newblock Edited by Richard Wentworth.

\bibitem[TV1]{ToVi16}
D.~Tonkonog and R.~Vianna.
\newblock {In preparation}.

\bibitem[TV2]{ToVi15}
D.~Tonkonog and R.~Vianna.
\newblock {Low-area Floer theory and non-displaceability}.
\newblock {\em arXiv:1511.00891}, 2015.

\bibitem[Via1]{Vi13}
R.~Vianna.
\newblock {On exotic Lagrangian tori in $\CP^2$}.
\newblock {\em arXiv:1305.7512}, 2013.

\bibitem[Via2]{Vi14}
R.~Vianna.
\newblock {Infinitely many exotic monotone Lagrangian tori in $\CP^2$}.
\newblock {\em arXiv:1409.2850}, 2014.

\bibitem[Wu]{Wu15}
W.~Wu.
\newblock {On an exotic Lagrangian torus in $\mathbb{C}P^{2}$}.
\newblock {\em Compositio Math.}, 151(7):1372--1394, 2015.

\end{thebibliography}
  
\end{document}